%% file: main.tex
\makeatletter
\def\input@path{{scheme-listings/}}
\makeatother

\documentclass[oneside]{article}

\newif\ifdraft
\drafttrue
\draftfalse

\newif\ifarxiv
\arxivtrue

\newlength\titlebox 
\usepackage[round]{natbib}

\usepackage{dsfont}
\usepackage[usenames,dvipsnames,svgnames,table]{xcolor}
\usepackage{cancel}
\usepackage{graphicx}
\usepackage{tabularx}
\usepackage{booktabs}
\usepackage{siunitx}

\usepackage{pgfplots}

\usetikzlibrary{quotes,angles,decorations.markings,patterns}

\ifdraft\ifarxiv\else
\fi\fi

\input{commands}
\input{notation}

\title{Line Search for Convex Minimization}
\author{Laurent Orseau, Marcus Hutter \\
Google DeepMind \\
\small \{lorseau,mhutter\}@google.com}
\ifdraft \date{January 2023} \else \date{} \fi

\begin{document}

\maketitle

\begin{abstract}
    Golden-section search and bisection search are the two main principled algorithms for 1d minimization of quasiconvex (unimodal) functions.
    The first one only uses function queries, while the second one
    also uses gradient queries.
    Other algorithms exist under much stronger assumptions, such as Newton's method for twice-differentiable, strongly convex functions with Lipschitz gradients.
    However, to the best of our knowledge, there is no principled exact line search algorithm for general \emph{convex} functions --- including piecewise-linear and max-compositions of convex functions --- that takes advantage of convexity.
    We propose two such algorithms: \gradalg{} is a variant of bisection search that uses (sub)gradient information and convexity to speed up convergence,
    while \nogradalg{} is a variant of golden-section search and uses only function queries.
    
    Both algorithms are based on a refined definition of the \emph{optimality region} $\Delta$ containing the minimum point, for general convex functions.
    
    While bisection search reduces the $x$ interval by a factor 2 at every iteration,
    \gradalg{} reduces the (sometimes much) smaller $x^*$-gap $\Delta^x$ 
    (the $x$ coordinates of $\Delta$)
    by \emph{at least} a factor 2 at every iteration.
    Similarly, \nogradalg{} also reduces the $x^*$-gap by at least a factor 2 every second function query.
    
    Moreover, and possibly more importantly, the $y^*$-gap $\Delta^y$ (the $y$ coordinates of $\Delta$) also provides a refined stopping criterion, which can also be used with other algorithms.
    
    Experiments on a few convex functions confirm that our algorithms are always faster than their quasiconvex counterparts, often by more than a factor 2.
    
    We further design a \emph{quasi-exact} line search algorithm based on \nogradalg{}.
    It can be used with gradient descent as a replacement for backtracking line search, for which some parameters can be finicky to tune --- and we provide examples to this effect, on strongly-convex and smooth functions.
    We provide convergence guarantees, and confirm the efficiency of quasi-exact line search on a few single- and multivariate
    convex functions.\footnote{
    One implementation of \nogradalg{} can be found at \url{https://github.com/deepmind/levintreesearch_cm/blob/main/lts-cm/delta-secant.rkt}}
\end{abstract}

\section{Introduction}

We consider the problem of exact function minimization in one dimension for general convex functions that are expensive to evaluate~\citep{boyd2004convex}.
This can occur for example when looking for a step size for gradient descent or Frank-Wolfe~\citep{frank1956quadratic} for logistic regression with millions of parameters.

While a backtracking line search is often used for this  problem~\citep{armijo1966minimization,boyd2004convex}, such algorithms are quite sensitive to their parameters, and there is no universal good choice.
In the past, exact line searches have been used, but they were deemed
too computationally costly.
In this work we propose a \emph{quasi}-exact line search for convex functions, which is both computationally efficient and has tight guarantees while being easy to tune --- that is, the default value of its single parameter appears to be a universally good choice. 

When gradient information is available and computationally cheap enough,
a classic algorithm for minimizing a 1-d \emph{quasi-convex} function is bisection search.
It ensures that the $x$ interval is halved at each iteration of the search.
We improve bisection search for 1-d \emph{convex} functions by taking advantage of convexity to refine the interval in which the optimum lies, which we call the optimality region $\Delta$.
We call this algorithm \gradalg{} and we show that it is provably faster than bisection search:
the size of the optimality interval $|\Delta^x|$ on the $x$ axis (which we call the $x^*$-gap) reduces by more than a factor 2 at each iteration, and often significantly more.
Moreover, for general convex functions we can use the optimality interval $|\Delta^y|$ (called the $y^*$-gap) to design a stopping criterion.
The $y^*$-gap is almost always bounded for convex functions, while it is almost always infinite for general quasi-convex functions.

When gradient information is not available or is too expensive to compute,
another classic algorithm to minimize 1-d quasi-convex functions is golden-section search (GSS).
It reduces the $x$ interval by a factor $(1-\sqrt{5})/2\approx 1.62$ at each iteration.
Again, we take advantage of convexity to refine the optimality region $\Delta(P)$, for a set $P$ of function evaluations.
We design an algorithm called \nogradalg{} that uses this information.
It is guaranteed to reduce the $x^*$-gap by a factor at least 2 every second iteration.
Note that the $x^*$-gap can be significantly smaller than the $x$ interval, 
so the guarantees of GSS and \nogradalg{} are not straightforwardly comparable.

Experiments on a few convex functions confirm that our algorithms are 
significantly faster than their quasi-convex counterparts.

Finally, we use \nogradalg{} to design a \emph{quasi-exact} line search algorithm to be combined with gradient descent-like algorithms 
for minimization of multivariate convex functions.
We prove a tight guarantee for strongly convex and smooth functions,
which is only a factor 2 worse than the guarantee for exact line search,
while being significantly faster in practice.
We validate on some experiments that our quasi-exact line search 
is competitive with a well-tuned backtracking line search, while being more robust and does not require to precisely tune a parameter.
Moreover, we demonstrate that backtracking line search has failure cases
even for smooth convex functions, making it non-straightforward to tune its parameters.
\todo{"And without the need to relax the convergence guarantee to enable faster convergence in practice"}

\section{Notation and Background}\label{sec:notation}

For a point $p^i$, we write $x^i$ and $y^i$ for its $x$ and $y$ coordinates,
 that is $p^i=(x^i, y^i)$.

For two points $p^i$ and $p^j$,
and define the secant line $f^{ij}(x)= a^{ij}(x-x^i)+y^i$ 
to be the line going through $p^i$ and $p^j$,
where  $a^{ij}=\frac{y^i-y^j}{x^i-x^j}$

For four points $p^i, p^j, p^k, p^l$,
\warning{$p^{ij,kl}$ to be consistent with \cref{fig:5point}}
define $p^{ij,kl}=(x^{ij,kl}, y^{ij,kl})$ to be the intersection point of the line going through $p^i$ and $p^j$ with the line going through $p^k$ and $p^l$. It can be checked that it can be expressed as
\begin{align}\label{eq:intersection}
    x^{ij, kl} &= \frac{f^{ij}(0) - f^{kl}(0)}{a^{kl}-a^{ij}}\,, &
    y^{ij, kl} &= f^{ij}(x^{ij, kl}) = f^{kl}(x^{ij, kl})
    = \frac{a^{kl}f^{ij}(0)-a^{ij}f^{kl}(0)}{a^{kl}-a^{ij}} \,.
\end{align}
There are many ways to express this point, but we choose this one as it is 
simple and is more closely related to the numerical expression we use for stability (see \cref{sec:stability}).

Let $P$ be a set of points,
then we define $\plow(P) = \argmin_{(x,y)\in P} y$,
and we also define $ (\ylow(P), \xlow(P)) = \plow(P)$,
and similarly for $\phigh$.

We assume that the user provides a convex function $f:[\xleft,\xright]\to\Reals\cup\{\infty\}$ and the interval $[\xleft,\xright]$. 
We write $x^*\in\argmin_{x\in[\xleft,\xright]} f(x)$ for some arbitrary minimum $x^*$ of $f$ within the given bounds,
and $y^* = f(x^*)$.

Given a convex function $f:[\xleft,\xright]\to\Reals$,
we say that a set of points $P$ is $(f,\xleft,\xright)$-convex
if for all $(x,y) \in P$, $y=f(x)$,
and $\argmin_{x\in[\xleft,\xright]} f(x)\subseteq [\xlow(P),\xhigh(P)]$.
(The last condition is easy to fulfill by including the points at $\xleft$ and $\xright$ in $P$, but, practically, it may be useful to remove unnecessary points in $P$.)

We also assume that the user provides a tolerance $\ytol$ 
and that the algorithms considered later return a point $(x, y)$ 
only when it can be proven that $y-y^*\leq \ytol$.

\todo{Define gradient-based bisection search, and golden-section search.}

Recall that a function $f:\Reals^d\to\Reals$ is $m$-strongly convex if,
for all $x, x' \in\Reals^d$:
\begin{align*}
    f(x') \geq f(x) + \innerprod{x' -x, \nabla f(x)} + \frac m2 \|x' -x\|_2^2\,,
\end{align*}
and $f$ is $M$-smooth if
for all $x, x' \in\Reals^d$:
\begin{align*}
    f(x') \leq f(x) + \innerprod{x' -x, \nabla f(x)} + \frac M2 \|x' -x\|_2^2\,.
\end{align*}

\section{Algorithms for 1-d convex minimization}

We first consider the case where gradient information is used,
and define the \gradalg{} algorithm.
Thereafter we consider the case where gradient information is not used,
and define the \nogradalg{} algorithm.

\subsection{\protect\gradalg: Convex line search with gradient information}

It is well-known that 1-d convex minimization can be performed 
using the bisection method to guarantee a reduction of the $x$ interval by a factor 1/2 at each iteration. Quoting \citet{boyd2004convex}:
\begin{quote}
    The volume reduction factor is the best it can be: it is always exactly 1/2.
\end{quote}
For general quasi-exact functions, only the sign of the gradient is informative --- and provides exactly 1 bit of information.
However, more information is available for convex functions, namely, the magnitude of the gradient.
Thus, it is almost always possible to reduce the interval by a factor less than 1/2, using nothing more than gradient information and convexity, and without performing more function or gradient queries than the original bisection algorithm.
For this purpose, we propose the \gradalg{} algorithm.
To show how this is achieved, we need to consider a refinement of the interval where the optimum $x^*$ can be located.
See \cref{fig:xgap_gradient}. 

\begin{figure}[tbh!]
    \centering
    \includegraphics[width=\textwidth]{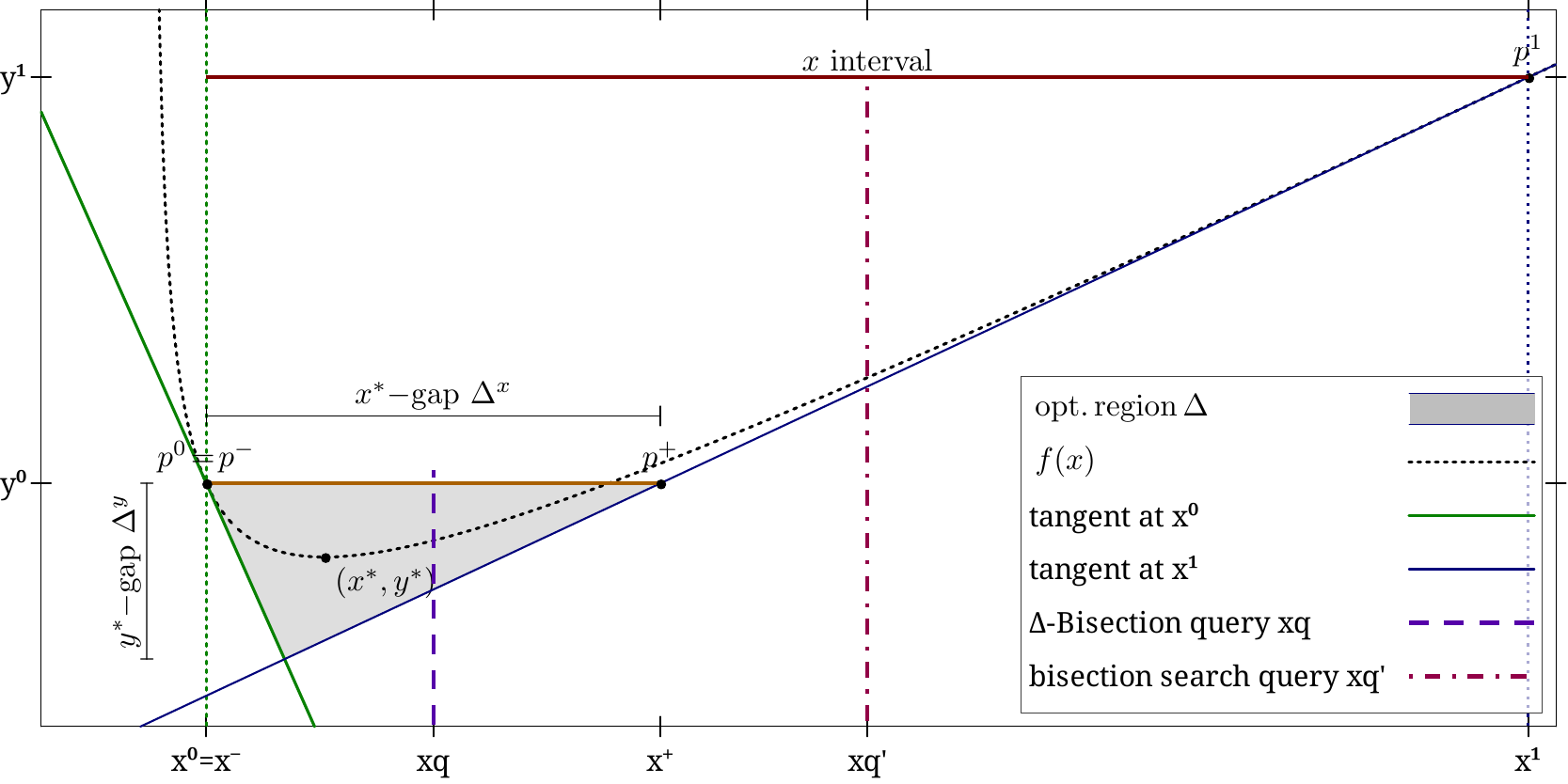}
    \caption{{\bf \gradalg:}
    An example where the $x^*$-gap
    $\Delta^x= [x^-, x^+] \approx [0.2, 1.85]$
    is significantly smaller (almost by a factor 3 here) than
    the $x$ interval
    $[x_0,x_1]=[0.2, 5]$,
    after querying the function $f$ and the gradients at $x^0$ and $x^1$.    
    $x^+$ is the $x$ value of the intersection of the tangent at $x^1$
    with the horizontal axis at $\min\{y^0, y^1\}$.
    Similarly, $x^-$ is the $x$-value of the intersection of the tangent at $x^0$
    with the horizontal axis at $\min\{y^0, y^1\}$ which, here, is $x^0$.    
    Due to convexity, we know that the minimum $x^*$ cannot be outside $[x^-, x^+]$.
    The original bisection search algorithm would query at
    the middle of $[x_0, x_1]$ at $x^{q'}=2.6$
    --- which is outside $[x^-, x^+]$ ---
    which, while reducing the $[x^0, x^1]$ interval by a factor 2, would reduce the $\Delta^x$ interval only marginally
    \mh{??}\comment{I'm not sure how to express this better}
    (using the tangent information at $x^{q'}$ --- not shown).
    By contrast, the \gradalg{} algorithm queries at $x^q$
    in the middle of $\Delta^x$, ensuring a reduction of $\Delta^x$ by at least a factor 2: once the gradient at $x^q$ becomes known,
    we will know that $x^* \notin[x^q, x^+]$ on this figure, which gives a reduction of a factor 2,
    and the tangent at $x^q$ will be used to calculate the new values of $x^-$ and $x^+$.
    }
    \label{fig:xgap_gradient}
\end{figure}

At each iteration $t$, the algorithm maintains two points, $p^0_t=(x^0_t, y^0_t)$ and $p^1_t=(x^1_t, y^1_t)$,
with $x^0_t \leq x^1_t$,
and the gradients of $f$ at these points $f'(x^0_t)$ and $f'(x^1_t)$.
Define $\xhigh_t\in \{x^0_t,x^1_t\}$
such that $\yhigh_t = f(\xhigh_t) =\max\{f(x^0_t),f(x^1_t)\}$,
and $\xlow_t \in \{x^0_t,x^1_t\}\setminus\{\xhigh_t\}$ with $\ylow_t= f(\xlow_t)$.

Define $f^a_b(x)= f(x^a_b) + (x - x^a_b)f'(x^a_b)$ to be the tangent function
of $f$ at $x^a_b$.
Then $f^0_t(\cdot)$ and $f^1_t(\cdot)$ are the two tangents at the two points
$x^0_t$ and $x^1_t$.
Then by convexity we know that 
$f(x^*) \geq f^0_t(x^*)$  and $f(x^*) \geq f^1_t(x^*)$,
and we also know that $f(x^*) \leq \ylow_t$.
We use these constraints to define the \emph{optimality region}
$\Delta_t$ at iteration $t$:
\begin{align*}
    \Delta_t = \{(x,y)\in\Reals^2 \mid
    x \in [x^0_t, x^1_t]\ \land \ {}
    y \leq \ylow_t \ \land \ {}
    f^0_t(x) \leq y \land f^1_t(x) \leq y\}\,,
\end{align*}
We also define the $x^*$-gap $\Delta^x_t$ and the $y^*$-gap  $\Delta^y_t$:
\begin{align*}
    \Delta^x_t &= \{x\mid (x,\cdot)\in\Delta_t\}\,, &
    \Delta^y_t &= \{y\mid (\cdot, y)\in\Delta_t\}\,. \\
\end{align*}
Then necessarily $(x^*, f(x^*))\in\Delta_t$ for all $t$,
and in particular $x^*\in\Delta^x_t$ and $f(x^*)\in\Delta^y_t$.

Define 
$x^-_t = \inf \Delta^x_t$ and  
$x^+_t =\sup \Delta^x_t$.
Then $x^-_t$ (resp. $x^+_t$) is the intersection of $f^0_t(\cdot)$ (resp. $f^1_t(\cdot)$) with the horizontal line at $\ylow_t$, that is,
\begin{align}\label{eq:x-x+}
    \Delta^x_t &= [x^-_t, x^+_t]\,,&
    x^-_t &= x^0_t + \frac{y^0_t-\ylow_t}{-f'(x^0_t)} \,,&
    x^+_t &= x^1_t - \frac{y^1_t - \ylow_t}{f'(x^1_t)}\,.
\end{align}
Note that either $x^0_t=x^-_t$ or $x^1_t= x^+_t$.
Then the \gradalg{} algorithm queries in the middle of $\Delta^x_t$:
\begin{align*}
    x^q_t = \frac12(x^-_t + x^+_t)\,.
\end{align*}
As for bisection search, if $f'(x^q_t) < 0$, then $x^* > x^q_t$ and 
we set $x^0_{t+1} = x^q_t$ and $x^1_{t+1} = x^1_t$;
otherwise we set 
 $x^0_{t+1} = x^0_t$ and $x^1_{t+1} = x^q_t$ (the case illustrated in Figure~\ref{fig:xgap_gradient}.
The pseudocode is given in \cref{alg:gradient_line_search}.

The $y^*$-gap is used to stop the algorithm
when $|\Delta^y_t|\leq \ytol$ for some user-specified tolerance $\ytol$.
$\Delta^y_t$ is simply the $y$ interval between the intersection $Y_t$ of the two tangents, $f^0_t$ and $f^1_t$, and the horizontal line at $\ylow_t$:
\begin{align}\label{eq:ygap}
    \Delta^y_t &= [Y_t, \ylow_t]\,, &
    Y_t &=f^0_t(x) \text{ for $x$ s.t. } f^0_t(x)=f^1_t(x)  \notag\\
    &&Y_t&= \frac{y^1_t f'(x^0_t) - y^0_t f'(x^1_t) + f'(x^0_t)f'(x^1_t)(x^0_t - x^1_t)}{f'(x^0_t) - f'(x^1_t)}
\end{align}
 
Some experimental results can be found in \cref{tab:gradient_results}.
They show that \gradalg{} always outperforms the original bisection algorithm,
often by a factor 2 or more.

\begin{theorem}[\gradalg\ gap-reduction]
For \cref{alg:gradient_line_search}, 
assuming $\Delta^y_{t+1} > \ytol$,
at each iteration $t\geq 2$, the $x^*$-gap $\Delta^x_t=[x^-_t,x^+_t]$ reduces at least by a factor of 2:
\begin{align*}
\frac{|\Delta^x_{t+1}|}{|\Delta^x_t|}&= 
\frac12\cdot\frac{\ylow_{t+1}- z_t}{\max\{y^q_t,\ylow_t\} - z_t} ~\leq~ \frac12\,, &
z_t &= \fhigh_{t+1}(\xlow_{t+1})\,. 
\end{align*}
where $\fhigh_{t+1}$ is the tangent of $f$ at $\xhigh_{t+1}$.
\qedhere
\end{theorem}

\begin{figure}[tbph!]
    \centering
    \includegraphics[width=\textwidth]{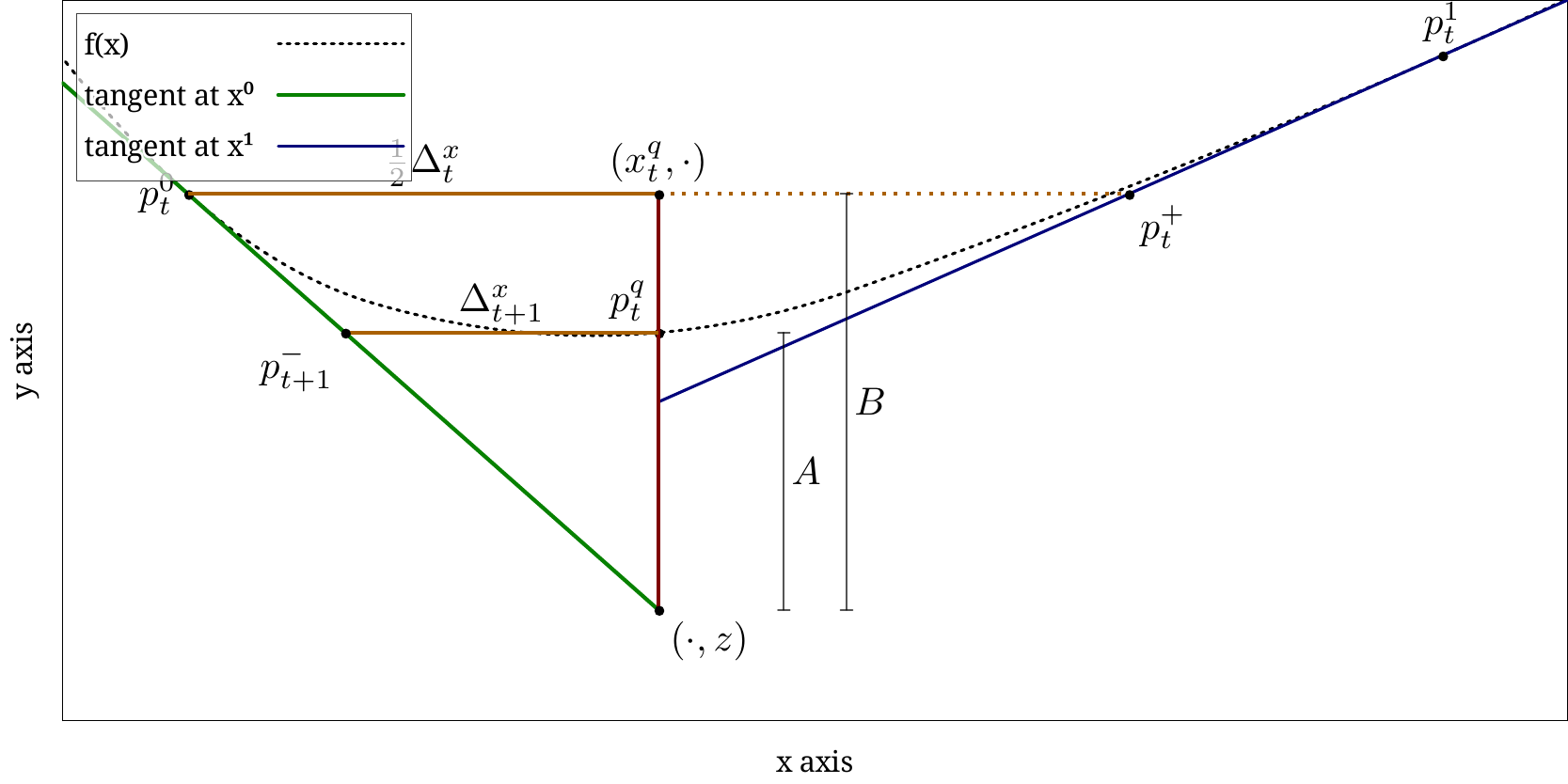}
    \caption{{\bf\boldmath \gradalg{} Case 1: $f'(x^q_t) > 0$ and  $y^q_t < y^0_t$,}
    that is, $x^* \in[x^0_t, x^q_t]$ or more precisely
    $x^* \in [x^-_{t+1}, x^q_t]$.
    The query point $(x^q_t, y^q_t)$ is denoted $q$ on the picture.
    Then $x^1_{t+1} = x^q_t = x^+_{t+1}$ and $x^0_t=x^0_{t+1}$.
    By proportionality, 
    $\frac{|\Delta^x_{t+1}|}{\frac12|\Delta^x_t|} = \frac AB = \frac{y^{q}_t - z}{y^{0}_t - z} = \frac{\ylow_{t+1} - z}{\ylow_t - z}$
    with 
    $z = f^0_t(x^q_t) = \fhigh_{t+1}(\xlow_{t+1})$.
    }
    \label{fig:case1}
\end{figure}

\begin{figure}[tbph!]
    \centering
    \includegraphics[width=\textwidth]{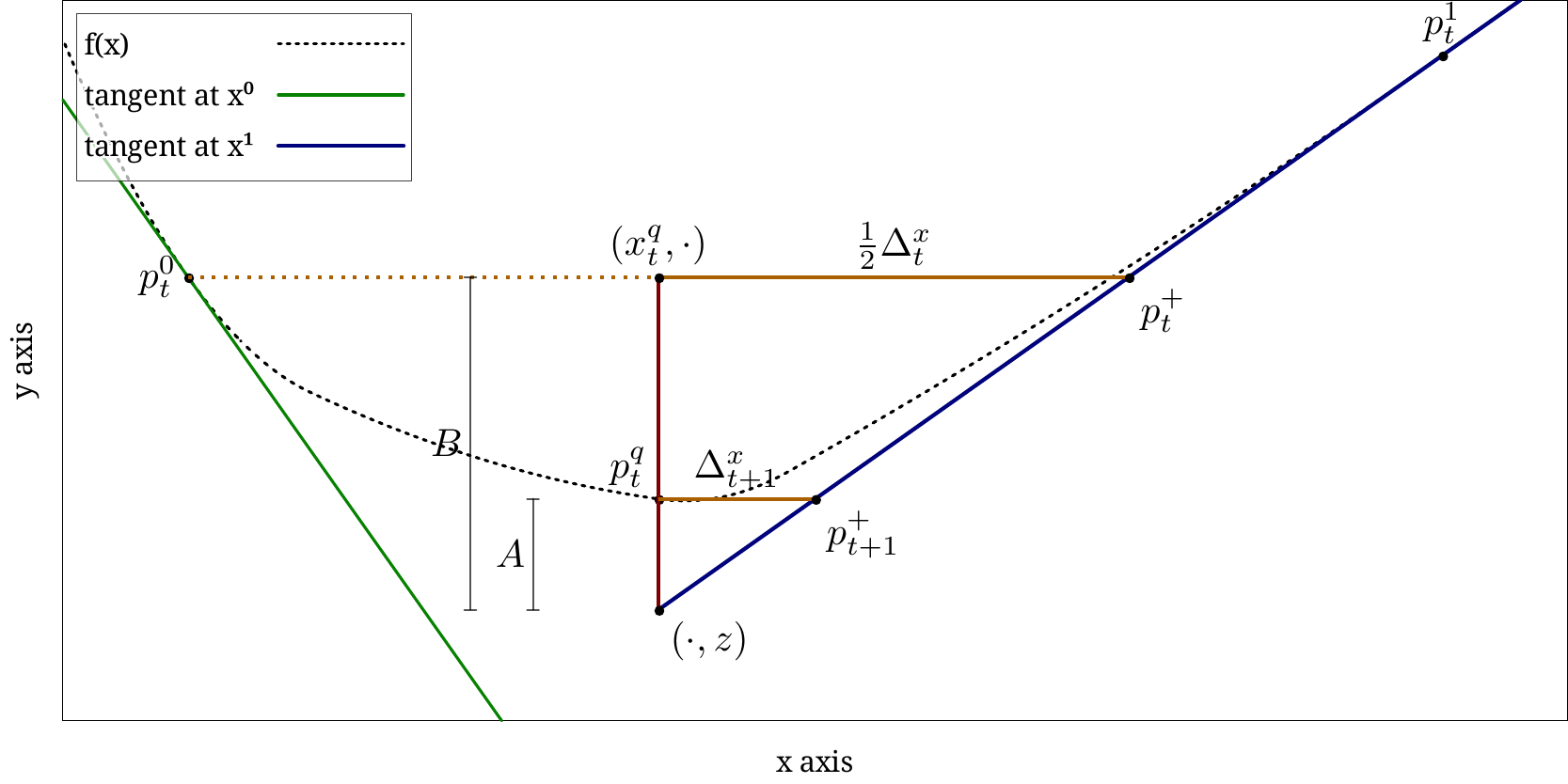}
    \caption{{\bf\boldmath \gradalg{} Case 2: $f'(x^q_t) < 0$ and $y^q_t < y^0_t$,}
    that is, $x^* \in[x^q_t, x^1_t]$ or more precisely
    $x^* \in [x^q_t, x^+_{t+1}]$.
    Then $x^0_{t+1} = x^q_t = x^-_{t+1}$ and $x^1_t=x^1_{t+1}$.
    By proportionality, $\frac{|\Delta^x_{t+1}|}{\frac12|\Delta^x_t|} = \frac AB = \frac{y^{q}_t - z}{y^{0}_t - z} = \frac{\ylow_{t+1} - z}{\ylow_t - z}$
    with
    $z = f^1_t(x^q_t) = \fhigh_{t+1}(\xlow_{t+1})$.    
    }
    \label{fig:case2}
\end{figure}

\begin{figure}[tbph!]
    \centering
    \includegraphics[width=\textwidth]{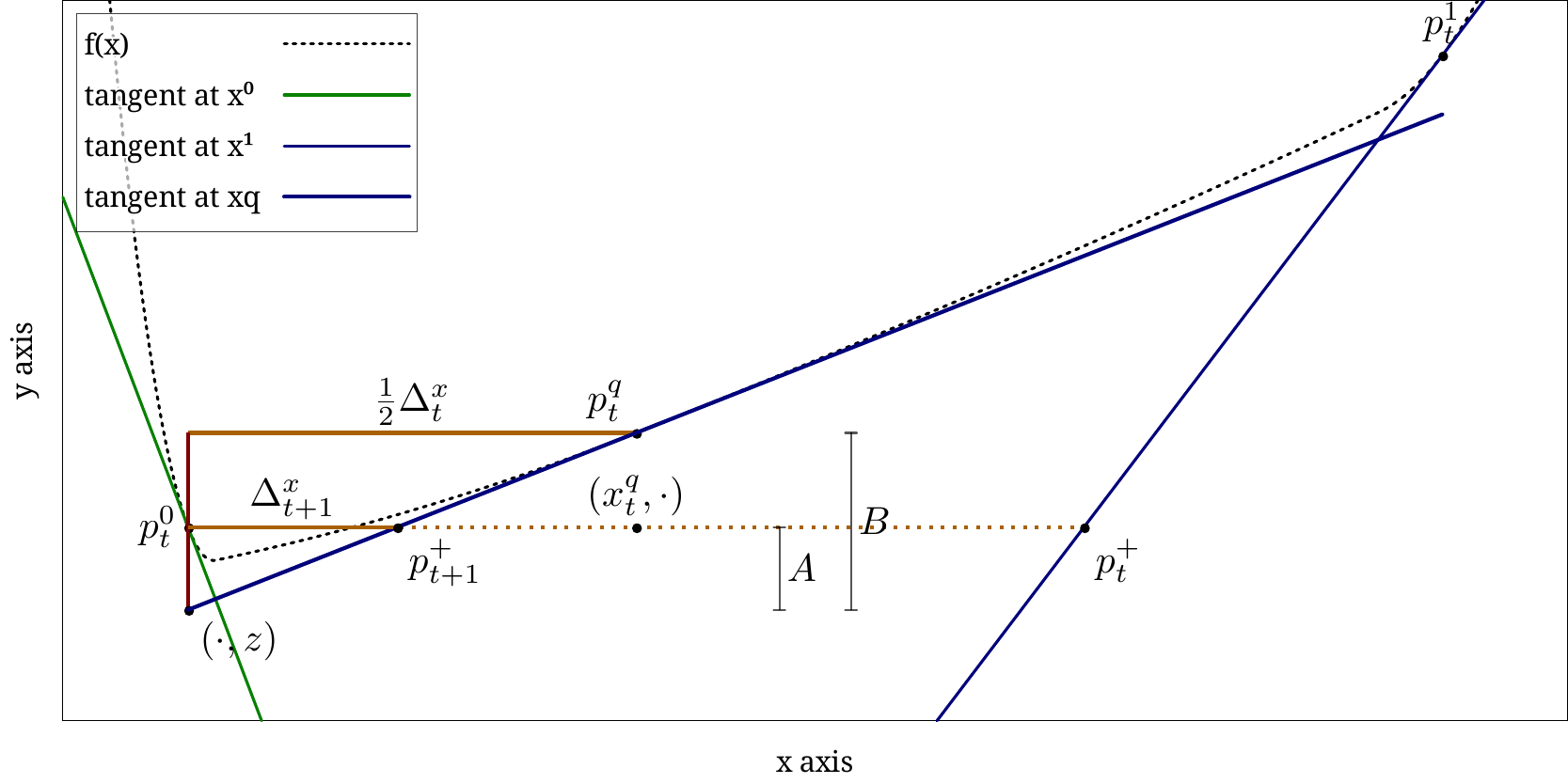}
    \caption{{\bf\boldmath \gradalg{} Case 3: $f'(x^q_t) > 0$ and $y^0_t < y^q_t < y^1_t$,}
    that is, $x^* \in[x^0_t, x^q_t]$ or more precisely
    $x^* \in [x^0_{t}, x^+_{t+1}]$).
    Then $x^1_{t+1} = x^q_t$ and $x^0_t=x^0_{t+1} = x^-_{t+1}$.
    Also note that $\ylow_t = \ylow_{t+1}=y^0_t$.
    By proportionality,
    $\frac{|\Delta^x_{t+1}|}{\frac12|\Delta^x_t|} = \frac AB = \frac{y^{0}_t - z}{y^{q}_t - z} = \frac{\ylow_{t+1} - z}{y^{q}_t - z}$
    with
    $z = f^q_t(x^0_t) = \fhigh_{t+1}(\xlow_{t+1})$.
    \todo{Remove x-axis and y-axis labels?}
    }
    \label{fig:case3}
\end{figure}

\newcolumntype{H}{>{\setbox0=\hbox\bgroup}c<{\egroup}@{}}
\begin{table}[tbph!]
    \centering
    \begin{tabular}{crr|rr}
    $f(x)$ & $\xleft$ & $\xright$ & Bisection & \gradalg{} \\
    \midrule
    $-x$                     & -20   & 7   & \textbf{0} & \textbf{0} \\
    $|x|$                    & -20   & 7   & 37         & \textbf{1} \\
    $\max\{-x, 2x\}$         & -20   & 7   & 37         & \textbf{19} \\
    $\max\{-x, 2x\}$         & -0.01 & 100 & 37         & \textbf{14} \\
    $|x|^{1.1}$              & -20   & 7   & 34         & \textbf{11} \\
    $x^2$                    & -20   & 7   & 20         & \textbf{12} \\
    $\sqrt{1+x^2}$           & -1000 & 900 & 27         & \textbf{19} \\
    $x\log x -x$             & 0.001 & 20  & 20         & \textbf{12} \\
    $\max\{x^2, (x-3)^2\}$   & -5    & 55  & 38         & \textbf{9} \\
    $\max\{x^2, (x/2-3)^2\}$ & -5    & 55  & 39         & \textbf{21} \\
    $x^4$                    & -20   & 7   & 11         & \textbf{8} \\
    $1/x^2 + x^2$            & 0.001 & 100 & 23         & \textbf{18}\\
    \end{tabular}
    \caption{{\bf
    Number of iterations (not counting the initial queries at the boundaries) required by bisection search and \gradalg{}
    for several convex functions},
    until provably reducing the $y^*$-gap $\Delta^y_t$ to at most a tolerance of $\ytol=10^{-10}$.    
    Observe that when the minimum is at a boundary, $\Delta^y_0=0$ and the algorithms terminate immediately.
    \gradalg{} is a little lucky for $f(x)=|x|$ because the $x^*$-gap is reduced to 0 on the second iteration, due to the symmetry of the function.
    }
    \label{tab:gradient_results}
\end{table}
\begin{table}[tbph!]
    \centering
    \begin{tabular}{crr|rr|rHr}
    $f(x)$ & $x_0$ & $x_1$ & \multicolumn{2}{c|}{w/ gradient queries} &\multicolumn{3}{c}{no gradient queries}\\
    &&& Bisection & \gradalg{}  & \nogradalg{} & halve area & GSS\\
    \midrule
    $-x$                     & -20   & 7   & 4  & 4           & \textbf{3}  & 3          & \textbf{3} \\
    $|x|$                    & -20   & 7   & 78 & \textbf{6}  &          7  & 6          & 50                    \\
    $\max\{-x, 2x\}$         & -20   & 7   & 78 & 42          & \textbf{23} & 24         & 50                    \\
    $\max\{-x, 2x\}$         & -0.01 & 100 & 78 & 32          & \textbf{18} & 19         & 52                    \\
    $|x|^{1.1}$              & -20   & 7   & 72 & \textbf{26} &          28 & 26         & 50                    \\
    $x^2$                    & -20   & 7   & 42 & 28          & \textbf{27} & 23         & 33                    \\
    $\sqrt{1+x^2}$           & -1000 & 900 & 58 & 42          & \textbf{23} & 22         & 41                    \\
    $x\log x -x$             & 0.001 & 20  & 44 & 28          & \textbf{23} & 21         & 32                    \\
    $\max\{x^2, (x-3)^2\}$   & -5    & 55  & 80 & 22          & \textbf{18} & 20         & 60                    \\
    $\max\{x^2, (x/2-3)^2\}$ & -5    & 55  & 82 & 46          & \textbf{26} & 25         & 58                    \\
    $x^4$                    & -20   & 7   & 26 & 20          & \textbf{18} & 10         & 21                    \\
    $1/x^2 + x^2$            & 0.001 & 100 & 50 & 40          & \textbf{31} & 31         & 36 \\
    \end{tabular}
    \caption{
    {\bf Number of \emph{queries}, 
    including the initial queries at the boundaries required by different algorithms}
    until provably reaching a $y^*$-gap of $|\Delta^y(P_t)|\leq \ytol=10^{-10}$,
    for several convex functions.
    For gradient-based algorithms, a function+gradient query counts as two queries.
    Note that our version of the golden-section search algorithm (GSS)
    also uses the 5-point $y^*$-gap $\Delta^y(P_t)$ as stopping criterion.
    }
    \label{tab:no_gradient_results}
\end{table}

\begin{proof}
For the inequality, observe that $\max\{y^q_t,\ylow_t\} \geq \ylow_t \geq \ylow_{t+1}$,
therefore $|\Delta^x_{t+1}|/|\Delta^x_t|\leq\frac12$,
since $\ylow_{t+1}>z_t$.

For the equality, first of all, if $f'(\xleft) \geq 0$ or $f'(\xright)\leq 0$, 
then the algorithm terminates immediately
since the $y^*$-gap $\Delta^y_t$ is zero.
Hence, in what follows we can safely assume that $f'(\xleft)<0$ and $f'(\xright)>0$, which implies that $f'(x^0_t) \leq 0$ and $f'(x^1_t) \geq 0$ for all $t$.

We focus on the case $y^0_t \leq y^1_t$. 
The case $y^0_t > y^1_t$ is handled similarly.

The result is shown by considering 3 separate cases.
The first case, $f'(x^q_t) > 0$ and  $y^q_t < y^0_t$,
is considered in \cref{fig:case1} (see caption).
The second case, $f'(x^q_t) < 0$ and $y^q_t < y^0_t$,
is considered in \cref{fig:case2}.
The last case, $f'(x^q_t) > 0$ and $y^0_t < y^q_t < y^1_t$,
is considered in \cref{fig:case3}.
All of these cases follow by using proportionality
and each figure shows that the result holds for each particular case.

Finally, the case $y_q > y^1_t$ is impossible by convexity,
as well as the case
$f'(y^q_t)>0$ and $y^q_t > y^0_t$.
\end{proof}

\begin{remark}[Root finding]
\gradalg{} can be adapted to improve upon bisection search for root finding of monotone increasing functions, if integral information is available.  
\end{remark}

\mh{Have you tried to use $x:f^0_t(x)=f^1_t(x)$ as the next query point in practice?}

\begin{algorithm}[tbph!]
\caption{{\bf (\gradalg{})} Line search for convex functions with gradient information.
\todo{inputs and outputs}
\todo{We can sometimes avoid a query if the end interval is already the minimum}
}
\label{alg:gradient_line_search}
\begin{lstlisting}[language=python,escapeinside={(*}{*)}]
def (*\gradalg{}*)($f, f', \xleft, \xright,$ y_tol):
  $x^0 = \xleft$, $y^0 = f(x^0)$, $g^0 = f'(x^0)$
  $x^1 = \xright$, $y^1 = f(x^1)$, $g^1 = f'(x^1)$
  for t = 1, 2, ...:
    if $|\Delta^y_t| \leq$ y_tol: return $P_t$ #  see (*\lstcommentcolor{\cref{eq:ygap}}*)
    $x^- = \inf \Delta^x_t$  # see (*\lstcommentcolor{\cref{eq:x-x+}}*)
    $x^+ = \sup \Delta^x_t$
    $x^q = (x^- + x^+)/2$
    $P_{t+1} = P_{t} \cup \{(x^q, f(x^q))\}$
\end{lstlisting}
\end{algorithm}

\subsection{\nogradalg: Convex line search without a gradient}

When gradient information is not available,
at the very least we can mimic the \gradalg{} algorithm
by querying twice at two very close points, which straightforwardly (disregarding numerical stability issues) gives
a reduction of a factor at least 2 every second query, that is, a reduction of a factor $\sqrt{2}$ of the $x^*$-gap at every iteration on average.
This would require exactly twice as many iterations as the gradient-based algorithm, but querying twice at the same point seems wasteful:
after the first one of two close queries, some information is already gained and the $x^*$-gap has likely already reduced, so why then query at almost the same point again?

We propose a no-gradient line search algorithm 
\nogradalg{}
for general convex functions that is also guaranteed to reduce the $x^*$-gap by a factor at least 2 at every second iteration,
but is likely to reduce it by more than this.
This is validated on experiments, which also show that \nogradalg{} is often significantly faster than golden-section search, and than \gradalg{} when counting a gradient query as a separate query.

First we need to define a version of the optimality region $\Delta(P)$ based only on a set of points $P$, without gradient information.

\subsubsection{Five-point optimality region}

Let $f^{ab}_t(x)$
be the line going through the points $p^a_t$ and $p^b_t$ (see \cref{sec:notation}).
Let $f$ be a convex function on $[\xleft, \xright]$, 
and let $p^a=(x^a, f(x^a)), p^b=(x^b, f(x^b))$,
then we know by convexity that not only $f(x) \leq f^{ab}(x)$ for all $x \in[x^a,x^b]$,
but just as importantly we also have
$f(x) \geq f^{ab}(x)$ for all $x\notin[x^a,x^b]$.
Let $P$ be a $(f,\xleft, \xright)$-convex set of points.
Recall that $x^*\in \argmin_{x\in[\xleft,\xright]} f(x)$.
Then, we use these convexity constraints to define the \emph{optimality region}
$\Delta(P)$ that contains $(x^*, f(x^*))$,
together with $f(x^*) \leq \ylow(P)$ (see \cref{sec:notation}):
\begin{align}\label{eq:gap}
    \Delta(P) = \bigg\{(x,y)\in\Reals^2{}\mid{}&
    x\in\left[\min_{(x',\cdot)\in P} x', \max_{(x',\cdot)\in P} x'\right]\ \land \notag\\
    &y\leq \ylow(P)\ \land \notag\\
    &\forall p^i, p^j\in P, x\notin(x^i, x^j): f^{ij}(x) \leq y
    \bigg\}\,,
\end{align}
and we define the \emph{$x^*$-gap} $\Delta^x(P)$ and \emph{$y^*$-gap} $\Delta^y(P)$:
\begin{align}\label{eq:xygap}
    \Delta^x(P) &= \{x\mid (x,\cdot)\in\Delta(P)\}\,, &
    \Delta^y(P) &= \{y\mid (\cdot, y)\in\Delta(P)\}\,. 
\end{align}

\begin{example}
Take $f(x)=x$,
and $P = \{(0, 0), (\tfrac12, \tfrac12), (1,1)\}$
then $\Delta(P) = \{(0,0)\}$.
\end{example}

\begin{example}
Take $f(x)=|x|$,
and $P = \{(-10, 10), (-5, 5), (-1, 1), (1, 1), (5,5), (10,10)\}$
then $\Delta^x(P) = [-1, 1]$ and $\Delta^y(P) = [0, 1]$.
\end{example}

\begin{theorem}[Optimality region]\label{thm:gap}
Let $P$ be a $(f,\xleft,\xright)$-convex set of points.
Then,
\begin{align*}
    (x^*, f(x^*))\in \Delta(P) ~~\text{for all}~~ x^*\in\argmin_{x\in[\xleft,\xright]} f(x)\,.
    &\qedhere
\end{align*}
\end{theorem}
\begin{proof}
First, by definition of $x^*$ and $P$, we must have $f(x^*) \leq \ylow(P)$.
Additionally,
for all $(x^1, y^1)\in P$ and $(x^2, y^2)\in P$,
if $x^*\notin (x^1, x^2)$ then by convexity of $f$ we have that
$f(x^*) \geq f^{12}(x^*)$.
Therefore, by definition of $\Delta$ in \cref{eq:gap}, it follows that $(x^*, f(x^*))\in \Delta(P)$.
\end{proof}

If only two points are available, that is, $|P|=2$, then $\Delta^x(P) = [\xlow(P), \xhigh(P)]$,
and $\Delta^y(P) = (-\infty, \ylow(P)]$.
For $|P|\geq 3$, $\Delta^y(P)$ is (usually) not infinite anymore,
and the optimality region $\Delta(P)$ consists of at most two triangles, 
while both $\Delta^x(P)$ and $\Delta^y(P)$ are bounded intervals ---
see \cref{fig:5point}.
Also, it turns out that having more than 5 points does not help, because the constraints generated by the points further away from $\plow(P)$ provide only redundant information.

\begin{theorem}[Five points are enough]\label{thm:five_points}
Let $P=\{p^i\}_{i\in[|P|]}$ be a $(f,\xleft,\xright)$-convex set of points such that $x^1 \leq x^2\leq \dots$
Let $k$ be such that $p^k=\plow(P)$, and assume
\footnote{This assumption can always be made to hold by adding specific virtual points to $P$.\XXX\todo{See below, but the resulting set of points is not $(f,\xleft,\xright)$-convex anymore, but is $(\tilde f,\xlow(P),\xright(P))$-convex where $\tilde f$ is infinite at the boundaries,
but otherwise is like $f$}}
that $3\leq k\leq |P|-2$.
Define $P'$ to be the subset of $P$ containing five points centered around the minimum, that is,
\begin{align*}
    P' = \{p^{k-2}, p^{k-1}, p^{k}, p^{k+1}, p^{k+2}\}
\end{align*}
then $P'$ is $(f,\xleft,\xright)$-convex and 
\begin{align*}
    \Delta(P') = \Delta(P)\,.
    &\qedhere
\end{align*}
\end{theorem}
\begin{proof}
Observe that necessarily $\Delta^x(P)\subseteq [x^{k-1}, x^{k+1}]$.
Recall that $f^{kj}$ is the line going through $p^k$ and $p^j$.

First, We show that $p^{k-3}$ (if it exists) is superfluous.

For all $i\geq k+1$,
for all $x\in \Delta^x(P)$
then necessarily $x\in [x^{k-3}, x^{i}]$,
and thus the constraint on $f^{k-3,i}$ in \cref{eq:gap}
is vacuous.
Thus, all constraints in \cref{eq:gap} involving $f^{k-3,i}$ can be removed.

For all $i \leq k-1, i\neq k-3$,
by convexity
we have $f^{i, k-3}(x^{k-1}) \leq y^{k-1}$
and thus 
for all $x\geq x^{k-1}$,
we have $f^{i, k-3}(x) \leq f^{i, k-1}(x)$.
Hence,
for all $x\geq x^{k-1}$,
for all $y$, $f^{i, k-1}(x) \leq y$ implies that $f^{i, k-3}(x) \leq y$.
Thus all constraints involving $f^{i, k-3}$ in \cref{eq:gap}
for all $i \leq k-1, i\neq k-3$ are superfluous.

Similarly, the constraint involving $f^{k-3, k}$ is made redundant by the constraint involving $f^{k-1, k}$.

Therefore, all constraints in \cref{eq:gap} involving $p^{k-3}$ are superfluous, and thus $\Delta(P\setminus\{p^{k-3}\})=\Delta(P)$.

Recursively, we can thus eliminate from $P$ all points to the left of $p^{k-2}$,
and similarly we can eliminate all points to the right of $k+2$.
Finally, there remain only the points of $P'$.

The fact that $P'$ is $(f,\xleft,\xright)$-convex is now straightforward.
\end{proof}

For simplicity, in what follows we consider the set $P_t$ of all points 
queried up to step $t$,
but an algorithm should instead only maintain the 5 points around the minimum --- as per \cref{thm:five_points}.

\begin{example}[4-point vs 3-point $y^*$-gap]
Take $f(x)= |x|$,
and $P=\{(-a, a), (\eps, \eps), (1, 1)\}$ for $\eps \in (0, 1)$ and $a > 0$,
and take $Q= P \cup \{(-(a+1), a+1)\}$.
Note that the additional point in $Q$ is further away from the minimum 0 of $f$ than the points in $P$.
Then it is easy to check that
$|\Delta^y(P)| = a+\eps$,
while $|\Delta^y(Q)|=\eps$.
Hence, a 4-point $y^*$-gap can be arbitrarily more accurate than a 3-point one.
Also note that more points only constrain the optimality region more, that is,
$\Delta(P\cup \{p\}) \subseteq \Delta(P)$ for any point $p$
as long as $P\cup p$ satisfies the assumptions of \cref{thm:five_points}.
\end{example}

The optimality region $\Delta(\cdot)$ can be used with existing line-search algorithms, such as golden-section search for example, assuming the function to minimize is convex, stopping the line search when $|\Delta^y(P_t)|\leq \ytol$.

For a given set $P=\{p^i\}_{i\in \{0,1,2,3,4\}}$ of five points,
assuming $\plow(P) = p^2$,
similarly to \cref{eq:x-x+,eq:ygap} we have
\begin{align}\label{eq:five_points_Deltas_exprs}
    x^-&=\inf \Delta^x(P) = x^1 + \frac{y^2-y^1}{a^{01}}\,, &
    a^{01}&= \frac{y^0-y^1}{x^0-x^1}\,, \notag\\
    x^+&=\sup \Delta^x(P) = x^3 + \frac{y^2-y^3}{a^{34}}\,, &
    a^{34}&= \frac{y^3-y^4}{x^3-x^4}\,, \notag\\
    |\Delta^x(P)| &= x^+ - x^-\,, \notag\\
    |\Delta^y(P)| &= y^2 - \min\{y^{01,23}, y^{12,34} \}\,, 
\end{align}
where $x^-$ is the intersection of $f^{01}$ with the horizontal line at $y^2=\ylow(P)$,
$x^+$ is the intersection of $f^{34}$ with $y^2$.
See \cref{eq:intersection} for the definition of $y^{ij,kl}$ and
\cref{sec:stability} for numerical stability considerations.

\begin{remark}[Missing points]
If there are fewer than 2 points to the left and right of $\plow(P)$,
virtual points can be added to restore this property.
Assume $P=\{p^2, p^3, \dots ,p^{n-2}\}$ with $x^2 \leq x^3 \leq \dots \leq x^{n-2}$ is $(f,\xleft,\xright)$-convex.
Define $Q = P\cup \{p^0, p^1, p^{n-1}, p^n\}$ 
where $x^0 = x^2 - 2\eps$, $x^1 = x^2 - \eps$,
$x^{n-1}  = x^{n-2} + \eps$, $x^n = x^{n-2}+2\eps$,
and 
$y^0=y^1=y^{n-1}=y^{n-2}=\infty$
for some infinitesimally small $\eps$.
Then necessarily $\plow(P)$ is surrounded by 2 points on each side,
while $\Delta(Q)=\Delta(P)$ since no information regarding the location of the minimum is introduced,
since by assumption $x^* = \argmin_{x\in[\xleft,\xright]} \in\Delta(P)$.
Technically, $Q$ is \emph{not} $(f,\xleft,\xright)$-convex,
but instead $(\tilde f,\xlow(P),\xhigh(P))$-convex where $\tilde f = f$ except on the additional points of $Q$.
\end{remark}

\subsubsection{\nogradalg}

The \nogradalg{} algorithm is an
adaptation of \gradalg{}, using $\Delta(P_t)$ instead of $\Delta_t$.
See \cref{alg:five_point_naive}.
It may happen in rather specific situations that $\xlow(P_t)$ is exactly the midpoint of $\Delta^x(P_t)$, which means that the query $x^q_t$ would coincide with $\xlow(P_t)$.
This poses no issue with the theory below as we can simply take the limit toward the midpoint, in which case the secant of these two infinitesimally close points would give a gradient information at $\xlow(P_t)$.
However, it does pose numerical issues as the algorithm would instead loop forever.
This can be easily circumvented by moving the query slightly away
\footnote{Still for numerical stability, this `repulsion' should even be applied whenever $|\xlow(P_t) - \Delta^x(P_t)/2| <\eps$.
Querying the gradient at $\xlow(P_t)$, if available, would be even better.}
from $\xlow(P_t)$ by some factor $\eps$ of $|\Delta^x(P_t)|$.

\begin{algorithm}
\caption{{\bf (\nogradalg{} line search algorithm for 1-d convex minimization)}
The improvement of \cref{rmk:save_one_query} is not implemented.
We advise to take $\eps= 2^{-7}$.
Not that in practice at most 5 points around $\plow(P_t)$ actually need to be maintained in $P_t$, as per \cref{thm:five_points}.
We decompose \nogradalg{} into two functions because we will re-use \nogradalg{}\code{_core} elsewhere.
\mh{Explain $\eps$ business}
}
\label{alg:five_point_naive}
\begin{lstlisting}[language=python,escapeinside={(*}{*)}]
def (*\nogradalg{}*)($f, \xleft, \xright,$ y_tol):
  def stop_when($P$): return $|\Delta^y(P)| \leq$ y_tol
  $P = \{(\xleft,f(\xleft)), (\xright, f(\xright))\}$
  return (*\nogradalg{}*)_core($f, P$, stop_when)
  
def (*\nogradalg{}*)_core($f, P_1$, stop_when):
  for t = 1, 2, ...:
    if stop_when($P_t$): return $P_t$
    (*\warning{Use $t$ indices?}*)
    $x^- = \inf \Delta^x(P_t)$ # see (*\lstcommentcolor{\cref{eq:five_points_Deltas_exprs}}*)
    $x^+ = \sup \Delta^x(P_t)$
    $x^m = (x^- + x^+)/2$
    if $(x^q, \cdot)\in P_t$:
      $x^q = x^m + \eps(x^+ - x^-)$  # necessarily, $\lstcommentcolor{x^q \notin P_t}$ for $\lstcommentcolor{\eps \in (0, \tfrac12)}$
    else:
      $x^q = x^m$
    $P_{t+1} = P_{t} \cup \{(x^q, f(x^q))\}$ # simple but inefficient: some points can be removed
\end{lstlisting}
\end{algorithm}

\begin{remark}[Saving one query]\label{rmk:save_one_query}
One early  function query can sometimes be saved:
given the initial bounds $\xleft$ and $\xright$, 
after obtaining $f(\xleft)$ and $f(\xright)$, 
a query at $x^q_1=(\xleft+\xright)/2$ is always required
since the algorithm cannot terminate with only 2 points,
since $|\Delta^y(\cdot)| = \infty$ (by contrast to \gradalg).
However, if we query at $x^q_1$ \emph{before} querying at $\xright$
and if it happens that  $f(\xleft) \leq f(x^q)$,
then by convexity it is not useful to query at $\xright$.
\end{remark}

\begin{figure}
    \centering
    \includegraphics[width=\textwidth]{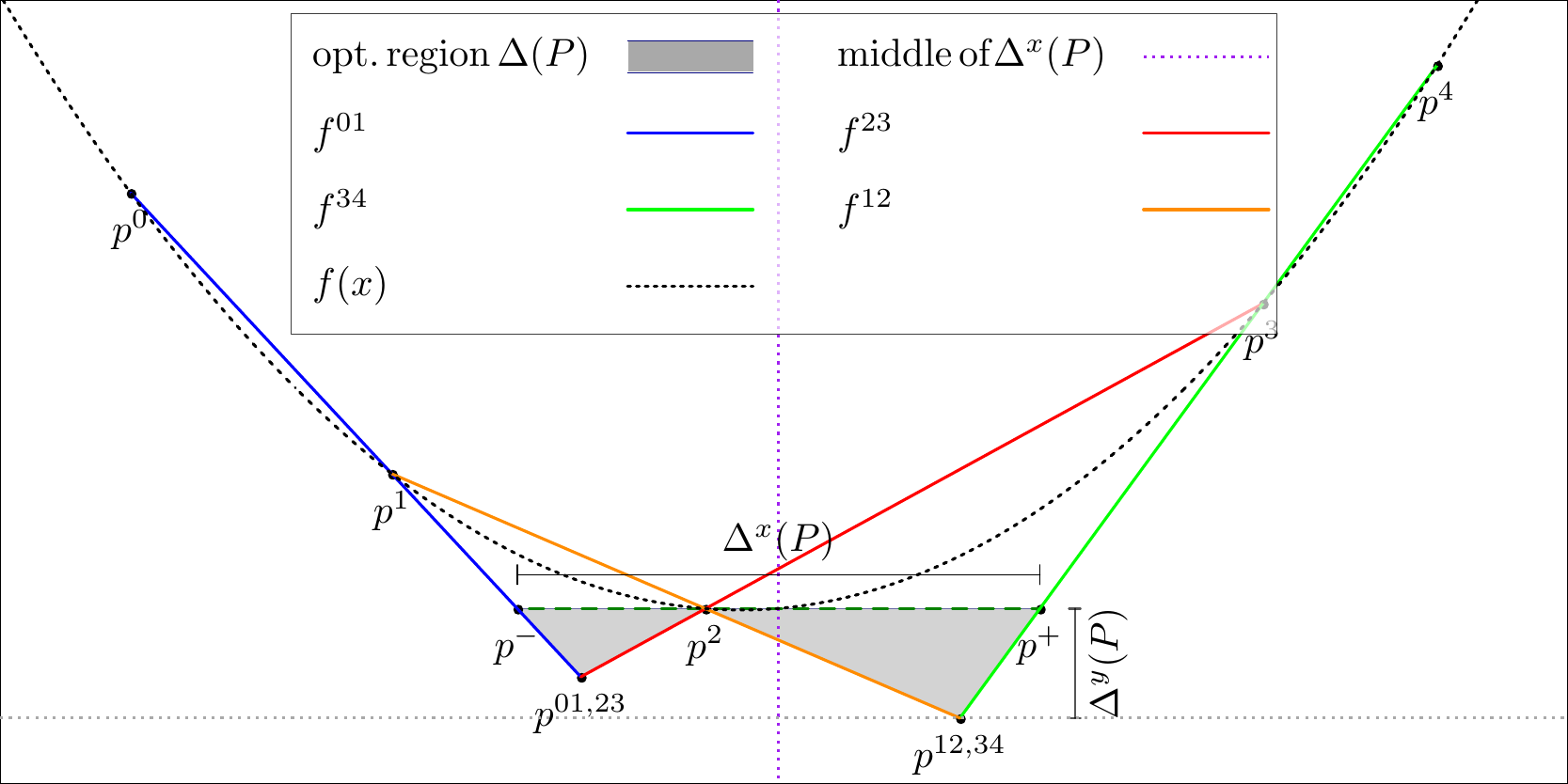}
    \caption{{\bf (\nogradalg)}
    The optimality region $\Delta$ based on 5 points
    $P=\{p^0, \dots p^4\}$ on a convex function $f$, without gradient information.
    The middle of $\Delta^x(P)$ corresponds to the query performed by \nogradalg.
    }
    \label{fig:5point}
\end{figure}

\begin{theorem}[\nogradalg{} $\Delta^x$ reduction]
For \nogradalg{} in \cref{alg:five_point_naive} where we take $\eps\to 0$,
for any iteration $t\in\Naturals$ such that $\Delta^y(P_{t+1}) > \ytol$ (\ie the algorithm does not terminate before $t+2$),
the $x^*$-gap $|\Delta^x(\cdot)|$ reduces by at least a factor 2 every second iteration:
for every $t\in\Naturals$, $|\Delta^x(P_{t+2})| \leq \frac12|\Delta^x(P_t)|$.
\end{theorem}
\begin{proof}
First, observe that for all $t\in\Naturals$, $\Delta(P_{t+1})\subseteq\Delta(P_t)$
since by definition of $\Delta$ in \cref{eq:gap}, $P_{t+1}$ is more constrained than $P_t$.
Also observe that $\plow(P_t) \in \Delta(P_t)$ (see \cref{fig:5point}).

Let $[x^-_t,x^+_t] = \Delta^x(P_t)$,
$x^q_t = \frac12(x^-_t + x^+_t)$ for all $t$.
\todo{figures}

Case 1) If $y^q_t \geq \ylow(P_t)$, 
then necessarily either $x^* \in [x^-_t, x^q_t]$ or $x^* \in [x^q_t,x^+_t]$
--- depending on which segment contains $\xlow(P_{t+1}) = \xlow(P_t)$ ---
and this information is reflected in $\Delta(P_{t+1})$\warning{Need to say how?}.
Thus $|\Delta(P_{t+1})| \leq \max\{x^q_t -x^-_t, x^+_t - x^q_t\} = \frac12|\Delta(P_t)|$.
Hence $|\Delta(P_{t+2})| \leq |\Delta(P_{t+1})| \leq \frac12|\Delta(P_t)|$.

Case 2) Assume instead that $y^q_t \leq \ylow(P_t)$.

Case 2a) At $t+1$, if $y^q_{t+1} \geq \ylow(P_{t+1}) = y^q_t$,
then by case 1) we have $|\Delta_{t+2}|\leq \frac12|\Delta_{t+1}|$,
and thus $|\Delta_{t+2}|\leq \frac12|\Delta_{t}|$.

Case 2b) If instead $y^q_{t+1} \leq \ylow(P_{t+1}) = y^q_t$,
then necessarily either $x^* \in [x^-_t, x^q_t]$ or $x^* \in [x^q_t,x^+_t]$
--- depending on which segment contains $\xlow(P_{t+2}) = x^q_{t+1}$ ---
and this information is reflected in $\Delta(P_{t+2})$.
Then again $\Delta(P_{t+2}) \leq \max\{x^q_t -x^-_t, x^+_t - x^q_t\} = \frac12|\Delta(P_t)|$.
\end{proof}

\warning{consistency: Golden-section search: caps? GSS? }
\newcommand{\Pgss}{P^{\text{GSS}}}
\begin{remark}[Golden-section search]
Let $\Pgss_t$ be the set of all points queried by Golden-section search
on a convex function $f$ up to step $t$.
Then, as mentioned earlier,
  Golden-section search is only guaranteed to reduce the $x$-\emph{interval}
  $x^{k-1}_t-x^{k+1}_t$ around the current minimum $x^k_t=\xlow(\Pgss_t)$
  (with $x^{k-1}_t, x^k_t, x^{k+1}\in \Pgss_t$)
  by a factor $(1+\sqrt{5})/2\approx 1.62$ at every iteration,
  but this does not necessarily lead to a strict reduction of the $x^*$-gap
  $\Delta^x(\Pgss_t)$
  at every (second) iteration.
  In fact, this may not even reduce $\Delta^x(\Pgss_t)$ for many iterations.  
  Consider for example the function $f(x)=\sqrt{1+x^2}$ for $[\xleft,\xright]=[-1000, 900]$.
  This function behaves like $|x|$ outside of $[-1, 1]$, and like $x^2$ inside $[-1, 1]$.
  After 4 iterations of \nogradalg{}, 
  the $x^*$-gap is reduced to less than 2,
  while it takes 13 iterations for GSS to reach the same $\Delta^x$ precision.
\end{remark}

\subsection{Quasi-exact line search for learning rate tuning}

In this section we consider optimizing a multidimensional function $f:\Reals^d\to\Reals$,
using a base algorithm such as Gradient Descent (GD) where the learning rate is tuned
at each step using a line search.
The typical line search algorithm for such cases is a backtracking line search with Armijo's stopping condition~\citep{armijo1966minimization},
which is called an inexact line search~\citep{boyd2004convex}.

\citet{boyd2004convex} mentions that people have stopped using exact line searches in favor of backtracking line search due to the former being very 
expensive.

Using the $y^*$-gap $\Delta^y$, we change the stopping criterion of \nogradalg{} so as to be suitable for tuning learning rates,
and prove convergence guarantees when combined with GD or Frank-Wolfe.
We call this type of line search algorithm a \emph{quasi-exact line search}.

We also show that Armijo's condition can lead to poor performance even on simple cases such as a quadratic and an exponential, and that neither large nor small values of the stopping criterion parameter are universally suitable.
By contrast, a quasi-exact line search appears significantly more stable for a wide range of values of its single parameter.

More precisely, at gradient descent iteration $k$, the current point being $x_k\in\Reals^d$,
define the 1-d convex function $\tilde f:[0,\infty)\to\Reals$:
\begin{align*}
    \tilde f(\alpha) = f(x_k - \alpha \nabla f(x_k))\,.
\end{align*}
It is known that there exists an $\alpha' > 0$ such that for all $\alpha\in[0,\alpha']$, $\tilde f(\alpha) < \tilde f(0)$ if the function is smooth (see \cref{sec:notation}).
The quasi-exact line search runs \nogradalg{} until the
$y^*$-gap (which is an upper bound on $\ylow(P_t) - y^*$)
is smaller than a constant fraction of the current improvement from the first queried point $\yleft(P_t) - \ylow(P_t)$:
\begin{align}\label{eq:quasi-exact-condition}
    \yleft(P_t) - \ylow(P_t) \geq c |\Delta^y(P_t)|
\end{align}
for the algorithm parameter $c > 0$.
See \code{quasi_exact_line_search} in \cref{alg:gd_quasi-exact}.
Moreover, if the point returned by \nogradalg{} is at the end of the original interval,
the interval size is increased by a factor 4  and \nogradalg{} is run again,
and so on until the returned point is not at the boundary. This scheme assumes that
$\tilde f(\alpha)$ eventually increases with $\alpha$, which applies in particular to strongly-convex functions (see \cref{sec:notation}).
\code{quasi_exact_line_search} can also use an initial query $\alpha_{\text{prev}}$ which defaults to 1.
We also propose a version of GD+Backtracking that also increases the interval size as required, see \cref{alg:gd_backtracking} --- also see Unbounded Backtracking~\citep{truong2021backtracking}.

We show in \cref{thm:quasi} that the terminating condition \cref{eq:quasi-exact-condition} incurs a slowdown
factor of $c/(c+1)$ compared to an exact line search ($c=\infty$),
in terms of the number of gradient descent steps.
Note that the conditions for \cref{thm:quasi} are weaker than smoothness and strong-convexity.

\begin{theorem}[Quasi-exact line search]\label{thm:quasi}
Let $\tilde f:[0,\infty)\to\Reals$ be a convex function.
Run \code{quasi_exact_line_search} (\cref{alg:gd_quasi-exact})
on $\tilde f(\cdot)$ and $c > 0$ and collect the return value $\alpha_c$.
Let $\alpha_\infty \in \argmin_{\alpha\geq 0}\tilde f(\alpha)$
and assume $\alpha_\infty < \infty$ and that $f(\infty) > f(\alpha_{\infty})$.
Then 
\begin{align*}
    \tilde f(0) - \tilde f(\alpha_c) \geq \frac{c}{c+1}\left(\tilde f(0) - \tilde f(\alpha_\infty)\right)\,. 
    &\qedhere
\end{align*}
\end{theorem}
\begin{proof}
First, since $\alpha_{\infty} < \infty$ and $f(\infty) > f(\alpha_{\infty})$, there there exists an iteration $i$ of \code{quasi_exact_line_search} at which \nogradalg{}\code{_core} returns with $\xlow_i(P) < \xright_i$.

Let $P$ be the points returned by \nogradalg{} at the last iteration of 
\code{quasi_exact_line_search}.
From \cref{thm:gap} we know that
$|\Delta^y(P)| \geq \ylow(P) - \tilde f(\alpha_\infty) = \tilde f(\alpha_c) - \tilde f (\alpha_\infty)$.
Thus by \cref{eq:quasi-exact-condition}, 
\begin{align*}
    \tilde f(0) -\tilde f(\alpha_c) ~\equiv~ \yleft(P) - \ylow(P) ~\geq~ c |\Delta^y(P)| ~\geq~ c(\tilde f(\alpha_c) - \tilde f(\alpha_\infty))\,.
\end{align*}
Adding $c\tilde f(0) -c\tilde f(\alpha_c)$ on each side we obtain
\begin{align*}
    (c+1)(\tilde f(0)-\tilde f (\alpha_c)) \geq c(\tilde f(0) -\tilde f(\alpha_\infty))
\end{align*}
and the result follows by dividing both sides by $c+1$.
\end{proof}

We now prove convergence guarantees for  Gradient Descent with a quasi-exact line search,
using a straightforward adaptation of the proof by \citet[Section 9.3.1]{boyd2004convex}.

\begin{algorithm}
\caption{{\bf (Gradient descent with quasi-exact line search)} for strongly convex and smooth functions for unbounded domains.
The line search is repeated with an increased range
as long as the minimum is at the right boundary.
This allows to adapt to the smoothness parameter $M$ without having to know it, even if $M < 1$.
Note that \nogradalg{} is very fast when the minimum is at a boundary, in particular since points of the previous iteration are reused.
The previously found learning rate $\alpha_{\text{prev}}$ is also used as a starting point for the line search.
See \cref{sec:stability} for  numerical stability.\mh{I would rename $\alpha_{\text{prev}}$ and $\alpha_c$ both to $\alpha$ and remove the last line of algorithm.}
\mh{quasi\_exact\_line\_search is supposed to return $\alpha_c$. Does it?}
\warning{$\alpha_prev$ used int tetxt too?}}
\label{alg:gd_quasi-exact}
\begin{lstlisting}[language=python,escapeinside={(*}{*)}]
def quasi_exact_line_search($\tilde f, c, \alpha_{\text{prev}}$):
  # stopping criterion of (*\lstcommentcolor{\cref{eq:quasi-exact-condition}}*)
  def stop_when(P): return $c|\Delta^y(P)| \leq \tilde f(0) - \ylow(P)$
  
  $P_0 = \{(0, \tilde f(0))\}$
  for i = 0, 1, 2, ...:
    $\xright = \alpha_{\text{prev}}4^i$
    # Re-use previous points and add the new upper bound, 
    # which forces $\lstcommentcolor{\Delta^x(P'_i) > \Delta^x(P_i)}$ by (*\lstcommentcolor{\cref{eq:gap}}*)
    $P'_i = P_i \cup \{(\xright, f(\xright))\}$
    
    $P_{i+1}$ = (*\nogradalg{}*)_core($\tilde f, P'_i$, stop_when)
    if $\xlow(P_i) < \xright$: # minimum not at the rightmost boundary
      return $\ylow(P_i)$ 

def GD_quasi_exact($f$, $x_0$, $c=1$): 
  $x_1$ = $x_0$ 
  $\alpha_{\text{prev}} = 1$ # previous learning rate
  for $k$ = 1, 2, ...:
    def $\tilde f(\alpha)$: return $f(x_k - \alpha\nabla f(x_k))$
      
    $\alpha_c$ = quasi_exact_line_search($\tilde f, c, \alpha_{\text{prev}}$)
    $x_{k+1} = x_k - \alpha_c \nabla f(x_k)$
    $\alpha_{\text{prev}} \leftarrow \alpha_c$
\end{lstlisting}
\end{algorithm}

\begin{algorithm}
\caption{{\bf (Gradient Descent with backtracking line search)} for strongly convex and smooth functions for unbounded domains
The line search is repeated with an increased range
as long as the minimum is at the right boundary.
This allows to adapt to the smoothness parameter $M$ without having to know it, even if $M < 1$.
The previously found learning rate $\alpha_{\text{prev}}$ is also used as a starting point for the line search.
See also \citet{truong2021backtracking}.}
\label{alg:gd_backtracking}
\begin{lstlisting}[language=python,escapeinside={(*}{*)}]
def backtracking($\tilde f, G, \tau, \eps, \alpha_{\text{prev}} $):
  $\alpha$ = $\alpha_{\text{prev}}$
  while True:
    i= 0
    while $\tilde f(0) - \tilde f(\alpha) < \eps \alpha G$:
      $\alpha \leftarrow \alpha\tau$
      i += 1
    if i == 1:  # $\alpha$ possibly too small
      $\alpha \leftarrow 4\alpha$
    else:
      return $\alpha$

def GD_backtracking($f$, $x_0$, $\tau=\frac12$, $\eps$): 
  $x$ = $x_0$
  $\alpha_{\text{prev}} = 1$
  for $k$ = 1, 2, ...:
    def $\tilde f(\alpha)$: return $f(x_k - \alpha\nabla f(x_k))$
      
    $\alpha$ = backtracking($\tilde f, \|\nabla f(x_k)\|_2^2, \tau, \eps, \alpha_{\text{prev}}$)
    $x_{k+1} = x_k - \alpha \nabla f(x_k)$
    $\alpha_{\text{prev}} \leftarrow \alpha$
\end{lstlisting}
\end{algorithm}

\begin{theorem}[GD with quasi-exact line search]
Consider an $m$-strongly convex and $M$-smooth function $f: \Reals^d\to \Reals$
with (unique) minimizer $x^*\in\Reals^d$.
Starting at $x_0\in\Reals^d$,
after $k$ iterations of gradient descent with quasi-exact line search
(\cref{alg:gd_quasi-exact}),
we have 
\begin{align*}
    f(x_{k+1}) - f(x^*) \leq (f(x_1) - f(x^*))\left(1-\frac{c}{c+1}\frac{m}{M}\right)^k\,.
    &\qedhere
\end{align*}
\end{theorem}
\begin{proof}
First, recall that strong convexity implies that the minimum $x^*$ is unique.
At iteration $k$, let 
$\alpha^\infty_k = \argmin_{\alpha \geq 0} f(x_k-\alpha \nabla f(x_k))$,
and let 
$x^\infty_k = x - \alpha^\infty_k \nabla f(x)$.
From \citet{boyd2004convex}, Chapter 9.3, at some gradient iteration $k$:
\begin{align*}
    f(x^\infty_k) \leq f(x_k) -\frac{1}{2M}\|\nabla f(x_k)\|^2\,.
\end{align*}
Together with \cref{thm:quasi},
\begin{align*}
    f(x_{k}) - f(x_{k+1}) \geq \frac{c}{c+1}\left(f(x_k) -f(x^\infty_k) \right)
    \geq \frac{c}{c+1}\frac{1}{2M}\|f(x_k)\|^2
\end{align*}
Rearranging and subtracting $f(x^*)$ on both sides,
\begin{align*}
    f(x_{k+1}) - f(x^*) \leq f(x_k) - f(x^*) - \frac{c}{c+1}\frac{1}{2M}\|f(x_k)\|^2
\end{align*}
Following \citet{boyd2004convex} again, we combine with $\|\nabla f(x_k)\|^2 \geq 2m(f(x_k) - f(x^*))$,
\begin{align*}
    f(x_{k+1}) - f(x^*) &\leq (f(x_k) - f(x^*))\left(1- \frac{c}{c+1}\frac{m}{M}\right) \\
    &\leq  \dots \\
    &\leq (f(x_1) - f(x^*))\left(1- \frac{c}{c+1}\frac{m}{M}\right)^k\,.
    \qedhere
\end{align*}
\end{proof}
Note that this implies 
\begin{align*}
f(x_{k+1}) - f(x^*)&\leq (f(x_1) - f(x^*))\exp\left(- k\frac{c}{c+1}\frac{m}{M}\right)\,, 
\end{align*}
which suggests that with the default value $c=1$,
the quasi-exact line search is at most a factor 2 slower (in the number of iterations $k$) than using an exact line search, while being significantly more efficient in the number of function queries.
 
\begin{remark}[Initial gradient]
\cref{alg:gd_quasi-exact} can also be modified to make use of the known initial gradient, by adding two virtual points at $-1$ and $-2$ for example.
The gradient information tells us that
\begin{align*}
    f(x - \alpha \nabla f(x)) \geq f(x) - \alpha \|\nabla f(x)\|_2^2\,,
\end{align*}
that is, $\tilde f(\alpha) \geq \tilde f(0) - \alpha \|\nabla f(x)\|_2^2$.
Hence, at gradient descent iteration $k$, we can augment the initial set of points $P_0$ of \code{quasi_exact_line_search} in \cref{alg:gd_quasi-exact}
with $\{(-2, f(x_k) + 2\|\nabla f(x_k)\|_2^2), (-1, f(x_k) + \|\nabla f(x_k)\|_2^2)\}$.
See \cref{sec:stability}, however.
\end{remark}

\subsection{Comparison with backtracking line search}

The backtracking line search with Armijo's stopping condition \citep{armijo1966minimization} is an appealing algorithm:
it is simple to implement, has good convergence guarantees, and works well in practice.
Armijo's condition stops the backtracking line search at the first iteration $t$ such that
\begin{align}\label{eq:armijo}
    f(x) - f(x - \tau^t\nabla f(x)) \geq \eps\tau^t \|\nabla f(x)\|_2^2
\end{align}
where $\tau\in(0,1)$ is the exponential decrease factor of the learning rate,
and $\eps\in(0,1)$ is the `acceptable linearity' parameter.

\citet{boyd2004convex} provide the following bound for GD with backtracking line search.
We slightly modify the traditional backtracking line search to extend the interval and start again whenever the returned value is at the boundary. See \cref{alg:gd_backtracking},
but compare \citet{truong2021backtracking} who define similar algorithms.
Then, if $f$ is an $m$-strongly convex and $M$-smooth function, then GD+backtracking converges linearly in unconstrained optimization:
\begin{align*}
    f(x_{k+1}) - f(x^*) \leq (f(x_1) - f(x^*))\left(1- 2\tau\eps\frac{m}{M}\right)^k
    \leq (f(x_1) - f(x^*))\exp \left(-2k\tau\eps\frac{m}{M}\right)\,.
\end{align*}
The proof is a simple adaptation of the standard proof~\citep{boyd2004convex}.
In what follows we always use $\tau=1/2$, which strikes a good balance between practical speed and theoretical analysis, as it leads to a learning rate that is always within a factor 2 of the optimal learning rate that satisfies Armijo's condition.
If $\alpha_\eps$ is the optimal learning rate for Armijo's condition for a fixed $\eps$, only $\log (1/\alpha_\eps)/\log (1/\tau)$ iterations are needed for the line search to stop within a factor $\tau$ of $\alpha_\eps$.
\warning{Yet again another notation for the optimal learning rate!}
While taking $\tau \ll 1/2$ may sometimes speed up the line search by a small log factor, it also degrades the guarantee by a much larger linear factor:
\eg if $\tau=1/1000$, this can speed up the search by a factor 
$\log 1000/\log 2 < 10$,
but the convergence speed guarantee degrades by a factor 500(!)
While setting $\tau=1/2$ is a good universal choice, finding the right value for $\eps$ is another matter altogether, as we show in this section.
The closer $\eps$ is to 1, the more we require the line between $(x_k, f(x_k))$ and 
$(x_{k+1}, f(x_k+1))$ to be close to the tangent of $f$ at $x_k$,
at GD iteration $k$.
This means that many line search iterations may be required to find a point that is close enough to the tangent, and in practice a small value of $\eps$ is chosen ---
\citet{boyd2004convex} mention $\eps\in[0.01, 0.3]$.
But even taking $\eps = 0.1$ `slows down' the convergence guarantee by a factor 10.

The following example shows (as is well-known) that small values of $\eps$ can make the algorithm sluggish even for quadratic functions --- the `nicest' class of functions for convex minimization, where $M=m$.

\begin{example}[Armijo's flying squirrel]\label{ex:armijos_flying_squirrel}
Take $f(x)= ax^2$.
Then,
\begin{align*}
    f(x) - f(x- \tau^t 2ax)
    = ax^2  - ax^2(1- \tau^t 2a)^2
    =  4a^2x^2\tau^t(1- \tau^t a) \,.
\end{align*}
Hence, Armijo's condition \cref{eq:armijo} is satisfied for the smallest integer $t \geq 0$ where
\begin{align*}
    4a^2x^2\tau^t(1- \tau^t a) &\geq \eps\tau^t \|\nabla f(x)\|^2 = 4\eps\tau^ta^2 x^2 \,,\\
    \Leftrightarrow
    1- \tau^t a &\geq \eps\,.
\end{align*}
and the condition must also \emph{not} be satisfied at $t-1$:
\begin{align*}
    4a^2x^2\tau^{t-1}(1- \tau^{t-1} a) &< \eps\tau^{t-1} \|\nabla f(x)\|^2 = 4\eps\tau^{t-1}a^2 x^2
    \,,\\
    \Leftrightarrow
    1- \tau^{t-1} a &< \eps\,.
\end{align*}
Take $a= \tau^{-t'}(1-\eps)$, with $t'\geq 0$ integer and $\eps \in (0, 1)$.
Then it can be verified that the two conditions above are satisfied for $t=t'$.
Furthermore, $x-\tau^t \nabla(x) = x- \tau^{t} 2ax = - x(1-\eps)$ for $t=t'$,
which means that at iteration $k$ of GD with backtracking,
we have $x_{k+1} = - (1-\eps)x_k$.
That is, $x_k$ constantly jumps between the left-hand side and right-hand side branches of the quadratic function --- like an over-excited flying squirrel --- getting closer to 0 only by factor $1-\eps$ at each iteration.

The number of iterations it takes before $|x_k| = |x_1|(1-\eps)^k \leq 1$
is linear with $1/\eps$ for $k \leq 1/\eps$.
This means that for small values of $\eps$, a GD with backtracking with parameters $\tau$ and $\eps$ takes linear time on the $x$-axis in $1/\eps$ to get close to 0.
For large $\eps$ such as $\eps=1/2$, this is all fine,
but for small $\eps$ such as $\eps=1/100$ this becomes unacceptable.
Note that the problem is not linked to $\tau$, or to the line search being bounded in $[0,1]$.
See \cref{tab:gd_line_search} for experimental results confirming this analysis.
\end{example}

The previous example shows that taking a small value for $\eps$ can lead to slow progress. The next example shows exactly the opposite: that large values of $\eps$ can also lead to slower progress than necessary ---
see also the example from Equation (9.20) from \citet{boyd2004convex}.

\begin{example}[Armijo's snowplow skier]\label{ex:armijos_sloth}
Consider $f(x)=e^{ax}$ for $a \geq 1$.
Define $g(x) = \nabla f(x)=a e^{ax}$, then the backtracking line search stops 
at step $t^*$ such that
(see \cref{eq:armijo})
\begin{align*}
    t^* &= \min\left\{t\in\Nonnegints:f(x) - f(x - \tau^tg(x)) \geq \eps \tau^t g(x)^2\right\} \\
    &= 
    \min\left\{t\in\Nonnegints:
    f(x)(1 - \exp(-a\tau^t g(x))) \geq \eps \tau^tg(x)^2
    \right\} \\
    &=\min\left\{t\in\Nonnegints:
    1 - \exp(-a\tau^t g(x)) \geq a \eps \tau^t g(x)
    \right\} \\
    &\geq \min\left\{t\in\Nonnegints:
    1  \geq a \eps \tau^t g(x)
    \right\}\\
    &= \max\left\{0, \ceil{\frac{ax+ \ln(a\eps)}{\ln 1/\tau}}\right\}\,.
\end{align*}
and since $1 \geq a \eps \tau^{t^*} g(x)$ (from line of the inequality in the above display), it follows that 
$\tau^{t^*}g(x) \leq 1/(a\eps)$, and so
$x - \tau^{t*}g(x) \geq x-1/(a\eps)$.
Moreover, each GD iteration $k$ takes at least $ax_k/\ln 1/\tau$ backtracking iterations
leading to at least a quadratic number $a^2\eps x^2/(2\ln 1/\tau)$ of queries in total!

On this very `curvy' function, Armijo's condition forces the backtracking line search to return a point that descends the slope extremely cautiously, despite the function being queried at much lower points during the line search.

For example, if $x_0 = 100$, $\eps=1/2$, $a=4$,
then it takes GD+backtracking more than 200 GD iterations before $x_k < 0$,
and a total of at least 55\,000 function queries.
By contrast to \cref{ex:armijos_flying_squirrel}, here a larger $\eps$ makes the matter worse.
\end{example}

See \cref{tab:gd_line_search} for
some experimental results, comparing backtracking line search and quasi-exact line search.

\newcommand{\bfred}[1]{\textcolor{red}{\textbf{#1}}}

\begin{table}[htb!]
    \centering
    \begin{tabular}{l|rr|rr||rr}
     & \multicolumn{2}{c|}{$f(x)=3.95x^2$}
     & \multicolumn{2}{c||}{$f(x)=e^{3x}+e^{-3x}$}
     & \multicolumn{2}{c}{$f(x,y)=x^4+y^4$} \\
    Line search type & GD iter. & Queries & GD iter. & Queries & GD iter. & Queries \\
    \hline
quasi-exact $c=100{}$       & 4             & 45    & 6             & 954     & 4             & 57      \\
quasi-exact $c=10{}$        & 4             & 35    & 7             & 957     & 4             & 49      \\
quasi-exact $c=4{}$         & 4             & 30    & 8             & 959     & 4             & 44      \\
quasi-exact $c=2{}$         & 4             & 25    & 8             & 955     & 5             & 44      \\
quasi-exact $c=1{}$         & 4             & 25    & 10            & 965     & 5             & 42      \\
quasi-exact $c=0.5{}$       & 4             & 25    & 10            & 965     & 5             & 38      \\
quasi-exact $c=0.1{}$       & 4             & 25    & 44            & 1090    & 7            & 43      \\
quasi-exact $c=0.01{}$      & 754           & 2265  & 44            & 1089    & 7            & 44      \\
backtracking $\eps=0.8{}$   & 67            & 476   & 898           & 193817  & 74 (17\,700)         & 388 (35662)   \\
backtracking $\eps=0.5{}$   & 4             & 25    & 229           & 48266   & 24 (17\,674)         & 118 (35416)   \\
backtracking $\eps=0.3{}$   & 4             & 25    & 147           & 31486   & 10 (17\,669)         & 68 (35\,376)   \\
backtracking $\eps=0.1{}$   & 4             & 25    & 47            & 9613    & 5 (17\,655)         & 39 (35\,327)   \\
backtracking $\eps=0.01{}$  & 754           & 3020  & 13            & 1376    & 8 (17\,662)         & 54 (35\,344)   \\
backtracking $\eps=0.001{}$ & 754           & 3020  & 10            & 867     & 8 (17\,662)         & 54 (35\,344)   \\
    \end{tabular}
    \caption{{\bf (Experimental results comparing backtracking line search and quasi-exact line search)}
    \todo{Use latest numbers. Reusing the previous learning rate helps}
    Number of GD iterations, and number of function+gradient queries for quasi-exact line search and backtracking line search ($\tau=1/2$) for various values of the terminating condition parameters $c$ and $\eps$, until $f(x_t) - f(x^*) \leq 10^{-10}$
    for different functions.
    For $f(x)=3.95x^2$ with $x_0=1000$, small values of $\eps$ make the backtracking line search slow (see \cref{ex:armijos_flying_squirrel}).
    \warning{Re-using the previous learning rate is far more efficient!}
    For $f(x)=e^{3x} + e^{-3x}$ with $x_0=100$, large values of $\eps$ make the backtracking line search slow (see \cref{ex:armijos_sloth}).
    By contrast, quasi-exact line search is fast for a wide range of values.
    For $f(x,y)=x^4+y^4$ with $(x_0,y_0)=(0.1, 15)$,
    for which the global minimum is degenerate (Example 3.7 from  \citet{truong2021backtracking}),
    standard backtracking line search (numbers in parenthesis) is stuck with a far too small learning rate,
    impeding progress. 
    Backtracking line search as in \cref{alg:gd_quasi-exact}
    and quasi-exact line search as in \cref{alg:gd_quasi-exact}
    avoid this particular issue.
    See also \citet{truong2021backtracking} for other variants.
    In line with \cref{thm:quasi}, taking $c=0.01$ for quasi-exact line search can incur a slowdown factor of $c/(c+1)\approx 0.01$ compared to an exact line search, and we can indeed observe on the first function that performance is poor.
    The table shows that there is no good default setting for backtracking line search.
    By contrast, quasi-exact line search is robust across a wide range of parameter values.
    We recommend taking $c=1$ to balance practical speed and tight theoretical guarantees.
    }
    \label{tab:gd_line_search}
\end{table}

\paragraph{Many-dimensional optimization.}
We also compare the algorithms on an artificial\mh{proto-typical} loss function \mh{used in Levin-Tree search?} of the form:
\begin{align}\label{eq:logistic_loss}
    L(x) = \left[\prod_{i=1}^N \frac{\exp(a_{i,n_i} x_{n_i})}{\sum_{j \in [d]} \exp(a_{i,j} x_j)}\right]^{-1}
\end{align}
with the parameters $x_0$ initialized uniformly randomly in $[-20, 20]^d$,
$a_{i,j}\in[-1,1]$ drawn uniformly randomly,
and $N=10$,
$d=100$,
and each $n_i\in[d]$ sampled uniformly.
We use Gradient Descent with line search and stop after 2000 (function+gradient) queries. 
\cref{fig:logistic_loss} shows the sensitivity of GD with Backtracking to the $\eps$ parameter, while GD with quasi-exact line search performs well a wide range of the parameter $c$.
Indeed, the exponential curvature makes it difficult for Armijo's condition with large parameter $\eps$ to return a value within only few iterations.

\begin{figure}
    \centering
    \includegraphics[width=\textwidth]{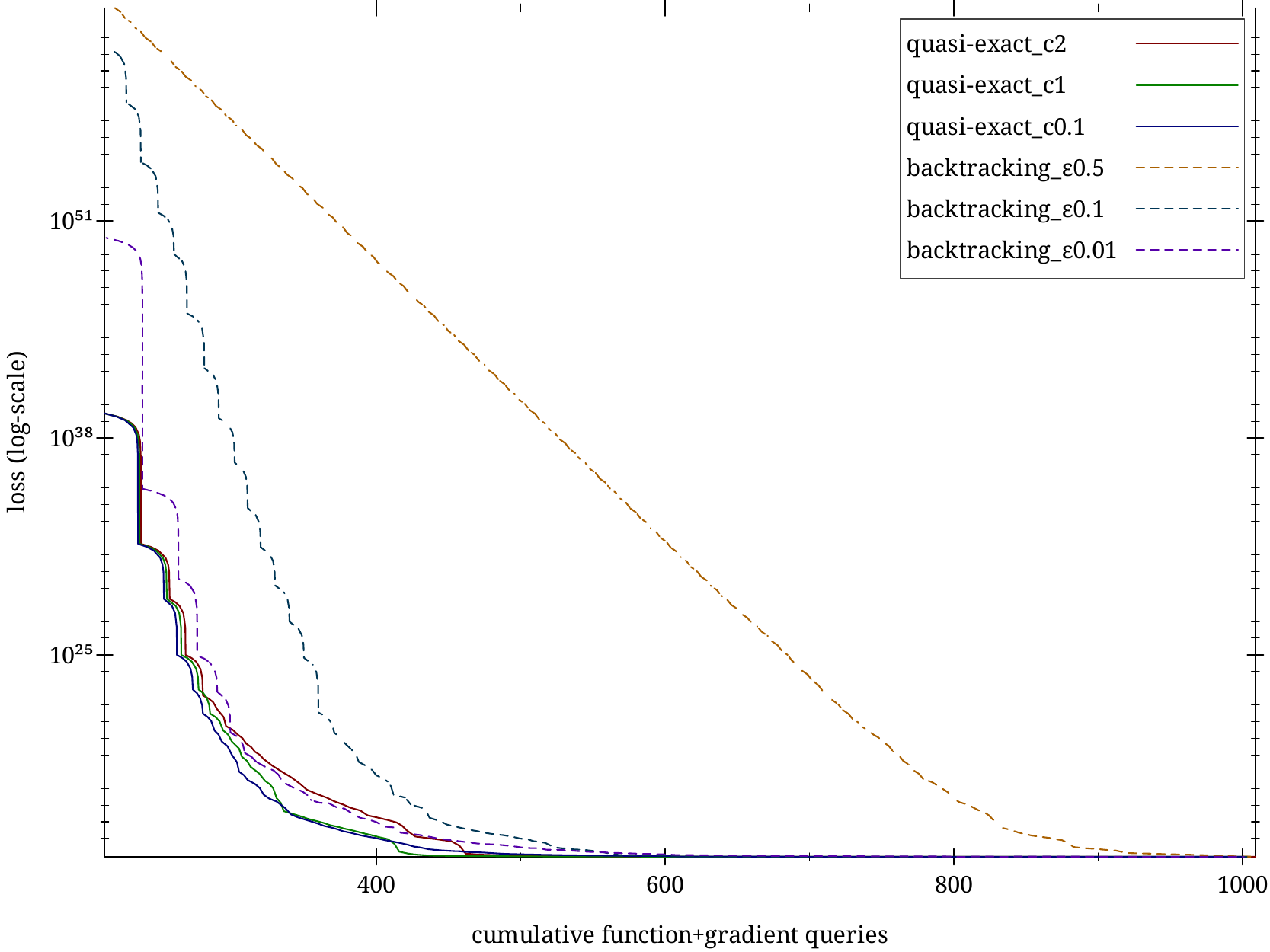}
    \caption{
    {\bf (Loss per function+gradient query for gradient descent with quasi-exact line search
    (\cref{alg:gd_quasi-exact})
    and backtracking line search
    (\cref{alg:gd_backtracking})
    for the loss function of \cref{eq:logistic_loss})}
    The $x$-axis is strongly correlated with computation time.
    This figure illustrates a typical example.
    Due to exponentials, $\eps\geq0.1$ is too conservative for backtracking line search (see \cref{ex:armijos_sloth}). 
    By contrast, quasi-exact line search is robust to various choices of $c$, and  the default $c=1$ appears to be a safe choice.
    }
    \label{fig:logistic_loss}
\end{figure}

\subsection{Frank Wolfe with quasi-exact line search}

\comment{See also ``How Does Momentum Help Frank Wolfe?'' \url{https://ieeexplore.ieee.org/stamp/stamp.jsp?tp=&arnumber=9457128}}
\comment{Also \url{https://proceedings.neurips.cc/paper/2015/file/c058f544c737782deacefa532d9add4c-Paper.pdf}}

A quasi-exact line search fits very well with Frank-Wolfe (FW) \citep{frank1956quadratic},
for a convex and smooth function $f:\mathcal{X}\to \Reals$
over a bounded convex domain $\mathcal{X}\subseteq\Reals^d$.
At FW iteration $k$, 
define $\tilde f(\alpha) = f((1-\alpha)x_k + \alpha s_k)$,
where $s_k = \argmin_{x\in\mathcal{X}} \innerprod{x,\nabla f(x_k)}$.
By contrast to \cref{alg:gd_quasi-exact} (with GD),
the range $[\xleft,\xright]$ should be fixed at $[0, 1]$ and not dynamically increased.

By making sure that the first line search query of iteration $k$ is made at $\alpha=2/(k+2)$,
the $O(1/k)$ convergence proof of Frank-Wolfe with a quasi-exact line search goes through straightforwardly~\citep{jaggi2013duality},
since the returned point will necessarily have a lower $y$-value than $\tilde f(2/(k+2))$.

Note that this is not the case for the learning rate that satisfies Armijo's condition in \cref{eq:armijo}.

\todo{Though an exact line search is likely to provide better guarantees 
in some circumstances.
Then we could provide a close bound.}

\section{Conclusion}

\gradalg{} and \nogradalg{} are two algorithms that perform
exact convex line search, the former using gradients while the latter uses only function queries.
These algorithms come with a guarantee on the optimality/duality gap.
We use this guarantee can as a stopping criterion to develop a quasi-exact line search and prove convergence guarantees that are within a factor $c/(c+1)$ of the convergence speed of an exact line search (for which $c=\infty$), while the slack given by $c$ allows to terminate the search much before reaching numerical precision.
\mh{Is really anybody using ``exact'' line search to a high precision as a sub-routine for multi-dimensional optimization?
This feels to me so dumb that nobody would even think about doing this, \emph{unless} gradient evaluations are much more expensive than 1d function/derivative computations. Shouldn't you stress the constant-factor gains of \gradalg{}/\nogradalg{} over bisection/secant?}

We recommend using \nogradalg{} for quasi-exact line search with for example $c=1$ to balance practical speed and theoretical guarantees.
\mh{Is it really a balance between theory and practice? If the total number of function calls are considered,
then theory \emph{and} practice require a moderate $c$, or not?}

Our algorithms compare favorably with bisection search, golden-section search and backtracking line search in their respective domains.

\bibliographystyle{unsrtnat}
\bibliography{biblio}

\begin{appendix}
\input{appendix}

\clearpage
\input{table_of_notation}
\end{appendix}

\end{document}

%% file: commands.tex
\usepackage{amsmath,amssymb}
\usepackage{amsthm}
\usepackage{ifthen}
\usepackage{hyperref}

\usepackage{bm}

\usepackage{algorithm}
\usepackage{subcaption}
\usepackage{listings}

\usepackage{letltxmacro}
\newcommand*{\SavedLstInline}{}
\LetLtxMacro\SavedLstInline\lstinline
\DeclareRobustCommand*{\lstinline}{
  \ifmmode
    \let\SavedBGroup\bgroup
    \def\bgroup{
      \let\bgroup\SavedBGroup
      \hbox\bgroup
    }
  \fi
  \SavedLstInline
}
\newcommand{\code}[1]{\texttt{\lstinline[mathescape,classoffset=1,keywordstyle=\color{black},basicstyle=\color{black},classoffset=0,keywordstyle=\color{black}]{#1}}}
\definecolor{commentcolor}{rgb}{0.133,0.545,0.133}
\newcommand{\lstcommentcolor}[1]{\textcolor{commentcolor}{#1}}

\usepackage[capitalize]{cleveref}
\crefname{listing}{Algorithm}{Algorithms}
\Crefname{listing}{Algorithm}{Algorithms}

\theoremstyle{definition}

\newcommand{\mynewtheorem}[3]{
\ifthenelse{\equal{#1}{theorem}}{
    \newtheorem{my#1}{#2}
}{
    \newtheorem{my#1}[mytheorem]{#2}
    \ifthenelse{\equal{#3}{}}{
        \crefname{my#1}{#2}{#2s}
    }{
        \crefname{my#1}{#2}{#3}
    }
}
}

\mynewtheorem{theorem}{Theorem}{}
\mynewtheorem{lemma}{Lemma}{}
\mynewtheorem{corollary}{Corollary}{Corollaries}
\mynewtheorem{definition}{Definition}{}
\mynewtheorem{assumption}{Assumption}{}
\mynewtheorem{proposition}{Proposition}{}
\mynewtheorem{remark}{Remark}{}
\mynewtheorem{example}{Example}{}
\mynewtheorem{exercise}{Exercise}{}

\newenvironment{theorem}
  {\pushQED{\qed}\mytheorem}
  {\popQED\endmytheorem}

\newenvironment{remark}
  {\pushQED{\qed}\myremark}
  {\popQED\endmyremark}
\newenvironment{example}
  {\pushQED{\qed}\myexample}
  {\popQED\endmyexample}

\definecolor{darkred}{rgb}{.5,0,0}
\definecolor{darkgreen}{rgb}{0,.5,0}
\definecolor{darkblue}{rgb}{0,0,.5}
\definecolor{darkorange}{rgb}{.8,.4,0}
\ifdraft
\newcommand{\todo}[1]{\textcolor{darkorange}{(\emph{TODO: #1})}}
\newcommand{\mh}[1]{\textcolor{darkblue}{(\emph{MH: #1})}}
\newcommand{\comment}[1]{\textcolor{gray}{(\emph{#1})}}
\newcommand{\warning}[1]{\textcolor{red}{(\emph{WARNING: #1})}}
\newcommand{\quest}[1]{\textcolor{darkgreen}{(\emph{Q: #1})}}
\newcommand{\XXX}{\textcolor{red}{\textbf{XXX}}}
\else
\newcommand{\todo}[1]{}
\newcommand{\mh}[1]{}
\newcommand{\comment}[1]{}
\newcommand{\warning}[1]{}
\newcommand{\quest}[1]{}
\newcommand{\XXX}{}
\fi

\newcounter{alphoversetcount}

\newcommand{\resetalph}{\setcounter{alphoversetcount}{0}}

\newenvironment{alphalign*}{
\csname align*\endcsname\resetalph{}
}{
\csname endalign*\endcsname\resetalph{}
}

\makeatletter
\def\th@plain{
  \thm@notefont{}
  \itshape
}
\def\th@definition{
  \thm@notefont{}
  \normalfont
}
\makeatother

%% file: notation.tex
\newcommand{\ie}{\emph{i.e.}, }
\newcommand{\eg}{\emph{e.g.}, }

\DeclareMathOperator*{\argmin}{argmin}

\newcommand{\Naturals}{{\mathbb N}_1}
\newcommand{\Nonnegints}{{\mathbb N}_0}
\newcommand{\Reals}{{\mathbb R}}

\newcommand{\innerprod}[1]{\langle #1 \rangle}

\newcommand{\ceil}[1]{\left\lceil#1\right\rceil}

\newcommand{\eps}{\varepsilon}

\newcommand{\xleft}{x^{\text{left}}}
\newcommand{\xright}{x^{\text{right}}}
\newcommand{\yleft}{y^{\text{left}}}
\newcommand{\yright}{y^{\text{right}}}
\newcommand{\xlow}{x^{\text{low}}}
\newcommand{\ylow}{y^{\text{low}}}
\newcommand{\plow}{p^{\text{low}}}
\newcommand{\xhigh}{x^{\text{high}}}
\newcommand{\yhigh}{y^{\text{high}}}
\newcommand{\phigh}{p^{\text{high}}}
\newcommand{\fhigh}{f^{\text{high}}}
\newcommand{\ytol}{y^{\text{tol}}}

\newcommand{\gradalg}{$\Delta$-Bisection}
\newcommand{\nogradalg}{$\Delta$-Secant}

%% file: appendix.tex
\section*{Appendix}

\section{A bad case for tangent intersection selection}

It is tempting to consider a variant of \cref{alg:gradient_line_search} where at each iteration we query at the intersection of the tangents of the two external points,
that is,
$x^q_t = \argmin_{(x,y)\in\Delta_t} y$.
This algorithm can be very slow, however.

Consider for example the function
$f(x)=\max\{-x/100, e^{bx}\}$ 
for $b > 0$, with $x^0_1=-100, x^1_1=100$.
It can be verified
that, for this algorithm,
for approximately $100/b$ steps $t$,
$x^0_t=x^0_1$ while $x^1_t \approx 100 - t/b$.
That is, only linear progress is made on the $x$ axis.

By constrast, \nogradalg{} takes less than 15 iterations for $b=10$ for example.

\section{Numerical stability}\label{sec:stability}

\subsection{Line-line intersection}
Some numerical stability care needs to be taken when implementing the line-line intersection function from 4 points to calculate the $y$-gap.

Empirically, we found that the following formula is more stable than the determinant method.
For 4 points $p^1, p^2, p^3, p^4$, 
with $y^1 \geq y^2$ and $y^3 \leq y^4$,
define 
$a^{12} = \frac{y^2-y^1}{x^2 - x^1}$
and $a^{34} = \frac{y^3-y^4}{x^3 - x^4}$,
and 
$f^{12}(x) = a^{12}(x - x^2) +y^2$ 
and $f^{34}(x) = a^{34}(x - x^3) +y^3$.

Then, for \emph{any} $\hat x\in\Reals$,
the intersection of $f^{12}$ and $f^{34}$ is at 
\begin{align*}
    x^c = \hat x + \frac{f^{12}(\hat x) - f^{34}(\hat x)}{a^{34} - a^{12}}
\end{align*}
First we pick $\hat x=x^2$, then we pick $\hat x = x^c$ and repeat the latter one more time to obtain a more accurate value. 
Finally, we return the last $x^c$, and we take $y^c = \min\{f^{12}(x^c), f^{34}(x^c)\}$
conservatively (for function minimization).
The quantity $|f^{12}(x^c) - f^{34}(x^c)|$ is the approximation error.

Additionally, (possibly numerical) zeros and infinities must be taken care of in a conservative manner (always underestimating $y^c$).

\subsection{Initial gradient for quasi-exact line search}
When using an initial gradient for quasi-exact line search, it is important to keep the line constraint
as a function $\alpha \mapsto f(x) - \alpha \|\nabla f(x)\|^2$ as is done in backtracking line search, rather than the more tempting solution of adding virtual points and letting the 5-point algorithm recover the tangent.

%% file: table_of_notation.tex
\section{Table of notation}
\begin{tabular}{c|l}
    $f$ & a convex function to minimize on $[\xleft, \xright]$ \\
    $f'$ & derivative of $f$ \\
    $p^a_t=(x^a_t, y^a_t)$ & a point labelled $a$ at iteration $t$\\
    $f^a_t(x)$ & tangent of $f$ at $x^a_t$ \\
    $f^{ab}_t(x)$ & line going through $p^{a}_t$ and $p^b_t$ \\
    $[\xleft, \xright]$ & initial interval in which to start the search \\
    $\yleft, \yright$ & $f(\xleft), f(\xright)$ \\
    $\xlow, \ylow$ & \XXX \\
    $\xhigh, \yhigh$ & \XXX \\
    $\ytol$ & User-specified tolerance used to terminate the line search \\
    $[x^0_t, x^1_t]$ & interval maintained by the gradient-based algorithms \\
    $x^q_t$ & query coordinate of an algorithm at iteration $t$ \\
    $\Delta_t$ & optimality region for \gradalg{} at step $t$; contains $(x^*, f(x^*))$\\
    $\Delta^x_t, \Delta^y_t$ & $x$ and $y$ coordinates of $\Delta_t$ \\
    $P_t$ & $(f,\xleft, \xright)$-convex set of points maintained by the \nogradalg{} algorithm\\
    $\Delta(P)$ & optimality region for a $(f,\cdot,\cdot)$ set of points $P$ \\
    $\Delta^x(P_t), \Delta^y(P_t)$ & Sets of $x$ and $y$ coordinates of $\Delta(P_t)$ \\
    $x^-, x^+$ & $\inf \Delta(P), \sup \Delta(P)$ \\
    $p^{ij,kl}=(x^{ij,kl}, y^{ij,kl})$ & Intersection point of $f^{ij}$ and $f^{kl}$ \\
\end{tabular}